\documentclass{amsart}

\usepackage[a4paper, margin=.8in]{geometry}

\usepackage[fontsize=13pt]{scrextend}

\usepackage{amsmath, amsthm, thmtools, amsfonts, amssymb}
\usepackage{enumitem}
\usepackage[dvipsnames]{xcolor}
\usepackage[pdfpagelabels, bookmarks, hyperindex, hyperfigures, colorlinks = true, linkcolor = {RoyalBlue}, citecolor = {BrickRed}, urlcolor = {Magenta}]{hyperref}
\usepackage{tikz-cd}

\renewcommand\emph[1]{\textcolor{darkgray}{\textit{#1}}}

\theoremstyle{plain}
\newtheorem{theorem}{Theorem}[section]
\newtheorem*{theorem*}{Theorem}
\newtheorem{theoremX}{Theorem}

\newtheorem{prop}[theorem]{Proposition}
\newtheorem{corr}[theorem]{Corollary}
\newtheorem*{corr*}{Corollary}
\newtheorem{lemma}[theorem]{Lemma}
\newtheorem{sublemma}[theorem]{Sublemma}
\theoremstyle{definition}
\newtheorem{defn}[theorem]{Definition}
\newtheorem{example}[theorem]{Example}
\newtheorem{obs}[theorem]{Observation}
\newtheorem*{obs*}{Observation}

\newtheorem*{question*}{Question}

\newtheorem*{conj*}{Conjecture}
\newtheorem{remark}[theorem]{Remark}

\DeclareMathOperator{\codim}{codim}
\DeclareMathOperator{\rk}{rk}
\DeclareMathOperator{\cork}{cork}
\DeclareMathOperator{\im}{Im}
\DeclareMathOperator{\bs}{bs}

\DeclareMathOperator{\diff}{Diff}
\DeclareMathOperator{\Op}{\mathcal{O}p}
\DeclareMathOperator{\Sym}{Sym}
\DeclareMathOperator{\Sol}{Sol}
\DeclareMathOperator{\Ann}{Ann}

\newcommand{\Id}{\textrm{Id}}
\newcommand{\Span}{\textrm{Span}}

\newcommand{\relHor}{\mathcal{R}^{\text{Hor}}}
\newcommand{\relHorTilde}{\tilde{\mathcal{R}}^{\text{Hor}}}
\newcommand{\relQHor}{\mathcal{R}^{\text{Hor}}}
\newcommand{\relQHorTilde}{\tilde{\mathcal{R}}^{\text{Hor}}}
\newcommand{\relCont}{\mathcal{R}^{\text{Cont}}}
\newcommand{\relContTilde}{\tilde{\mathcal{R}}^{\text{Cont}}}
\newcommand{\relICont}{\mathcal{R}^{\text{IsoCont}}}
\newcommand{\relIContTilde}{\tilde{\mathcal{R}}^{\text{IsoCont}}}

\newcommand{\solHor}{\Phi^{\text{Hor}}}
\newcommand{\solCont}{\Phi^{\text{Cont}}}
\newcommand{\solContTilde}{\tilde{\Phi}^{\text{Cont}}}

\newcommand{\solIContTilde}{\tilde{\Phi}^{\text{IsoCont}}}

\newcommand{\secCont}{\Psi^{\text{Cont}}}
\newcommand{\secContTilde}{\tilde{\Psi}^{\text{Cont}}}

\newcommand{\opCont}{\mathfrak{D}^{\text{Cont}}}
\newcommand{\opContTilde}{\tilde {\mathfrak{D}}^{\text{Cont}}}

\newcommand{\linCont}{\mathfrak{L}^{\text{Cont}}}

\newcommand{\resLinCont}{\mathcal{L}^{\text{Cont}}}

\begin{document}
\allowdisplaybreaks
\title{Existence of Horizontal Immersions in Fat Distributions}

\author[A. Bhowmick]{Aritra Bhowmick}
\address{A. Bhowmick, Department of Mathematics and Statistics, Indian Institute of Science Education and Research, Kolkata\\ Mohanpur, Nadia 741246, West Bengal, India}
\email{avowmix@gmail.com}

\author[M. Datta]{Mahuya Datta}
\address{M. Datta, Statistics and Mathematics Unit, Indian Statistical Institute\\ 203, B.T. Road, Kolkata 700108, India}
\email{mahuya@isical.ac.in}

\subjclass[2010]{58A30, 58D10, 58A20, 58C15}

\keywords{Horizontal immersions, Bracket generating distributions, Fat distributions, Holomorphic contact structures, Quaternionic contact structures, h-principle}

\begin{abstract}
Contact structures, as well as their holomorphic and quaternionic counterparts, are the primary examples of strongly bracket generating (or fat) distributions. In this article, we associate a numerical invariant to corank $2$ fat distribution on manifolds, referred to as \emph{degree} of the distribution. The real distribution underlying a holomorphic contact structure is of degree $2$. Using Gromov's sheaf theoretic and analytic techniques of $h$-principle, we prove the existence of horizontal immersions of an arbitrary manifold into degree $2$ fat distributions and the quaternionic contact structures. We also study immersions of a contact manifold inducing the given contact structure.
\end{abstract}

\maketitle

\frenchspacing

\section{Introduction}\label{sec:intro}
A \emph{distribution} on a manifold $M$ is a subbundle $\mathcal{D}$ of the tangent bundle $TM$. The sections of $\mathcal{D}$ constitute a distinguished subspace $\Gamma(\mathcal{D})$ in the space of all vector fields on $M$. On one end there are \emph{involutive} distributions for which $\Gamma(\mathcal{D})$ is closed under the Lie bracket operation, while at the polar opposite there are \emph{bracket-generating distributions} for which the local sections of $\mathcal{D}$ generate the whole tangent bundle under successive Lie bracket operations. A celebrated theorem due to Chow says that if $\mathcal{D}$ is a bracket-generating distribution on a manifold $M$, then any two points of $M$ can be joined by a $C^\infty$-path horizontal (that is, tangential) to $\mathcal{D}$ \cite{chowBracketGenerating}. This is the starting point of the study of subriemannian geometry. Chow's theorem is clearly not true for involutive distributions since by Frobenius theorem the set of points that can be reached by horizontal paths from a given point is a (integral) submanifold of dimension equal to the rank of $\mathcal{D}$.

Given an arbitrary manifold $\Sigma$ there is a distinguished class of maps $u:\Sigma\to (M,\mathcal{D})$ such that $T\Sigma$ is mapped into $\mathcal{D}$ under the derivative map of $u$. Such maps are called $\mathcal{D}$-horizontal maps or simply horizontal maps. If $u$ is an embedding then the image of $u$ is called a \emph{horizontal submanifold} in $(M,\mathcal{D})$. An immediate question that arises after Chow's theorem is the following: For a given distribution $\mathcal{D}$ on $M$ and a given point $x\in M$, what is the maximum dimension of a (local) horizontal submanifold through $x$? More generally, can we classify $\mathcal{D}$-horizontal immersions (or embeddings) of a given manifold into $(M,\mathcal{D})$ up to homotopy? These questions have been studied in generality by Gromov, and the answers can be given in the language of $h$-principle.

Horizontal immersions of a manifold $\Sigma$ in $(M,\mathcal{D})$ can be realized as solutions to a first order partial differential equation associated with a differential operator $\mathfrak{D}$ defined on $C^\infty(\Sigma, M)$ and taking values in $TM/\mathcal{D}$-valued $1$-forms on $\Sigma$. If $\mathcal{D}$ is globally defined as the common kernel of independent $1$-forms $\lambda^i$ on $M$ for $i=1,\ldots, p$, then the operator can be expressed as
$$\mathfrak{D} : u \mapsto \big(u^*\lambda^1,\ldots,u^*\lambda^p\big).$$ 
This operator is infinitesimally invertible on \emph{$\Omega$-regular} horizontal immersions (\autoref{defn:contOmegaRegular}), where $\Omega$ is the curvature form of $\mathcal{D}$. It follows from an application of the Nash-Gromov Implicit Function Theorem that $\mathcal{D}$ is locally invertible on $\Omega$-regular immersions. An integrable distribution $\mathcal{D}$ has vanishing curvature form; as a consequence, there are no $\Omega$-regular immersions. In order to have an $\Omega$-regular horizontal immersion, it is necessary that $k(p+1)\leq \rk\mathcal{D}$. Gromov proves that for \emph{generic} distribution germs this is also sufficient. Moreover, with sheaf theoretic techniques, he proves the $h$-principle for horizontal immersions satisfying the so-called `overregularity' condition, which demands that $(k+1)(p+1)\leq \rk\mathcal{D}$. Gromov, however, conjectures an $h$-principle for $\Omega$-regular horizontal immersions under the condition $k(p+1)< \rk\mathcal{D}$, the operator being underdetermined in this range. 

Among all the bracket-generating distributions, \emph{contact structures} have been studied most extensively \cite{geigesBook}. These are corank-$1$ distributions on odd-dimensional manifolds, which are maximally non-integrable. In other words, a contact structure $\xi$ is locally described as the kernel of a $1$-form $\alpha$ such that, $\alpha\wedge (d\alpha)^n$ is non-vanishing, where the dimension of the manifold is $2n+1$. Since $d\alpha$ is non-degenerate on $\xi$, the maximal dimension of a horizontal submanifold in $(M,\xi)$ is $n$. The $n$-dimensional horizontal submanifolds of a contact structure are called \emph{Legendrians}. Locally, there are plenty of $n$-dimensional horizontal (Legendrian) submanifolds, due to the existence of Darboux charts. Globally, Legendrian immersions and loose Legendrian embeddings are completely understood in terms of $h$-principle \cite{gromovBook,duchampLegendre,murphyLooseLegendrianThesis}. Any horizontal immersion in a contact structure is $\Omega$-regular by default. Moreover, one does not require the overregularity condition to obtain the $h$-principle.

Beyond the corank-$1$ situation, very few cases are completely understood. \emph{Engel structures}, which are certain rank-$2$ distributions on $4$-dimensional manifolds \cite{engelOriginal}, have been studied in depth in recent years, and the question of existence and classification of horizontal loops in a given Engel structure has been solved \cite{adachiEngelLoops,pinoEngelTangentTransverse,casalsPinoEngelKnot}. We refer to \cite{dAmbraSubbundle} for horizontal immersions in products of contact manifolds. In a recent article \cite{pinoKnotEmbeddingBracketGenerating}, the authors have obtained the complete $h$-principle for a regular class of immersed (and embedded) loops in any equiregular bracket-generating distributions of rank $\ge 3$.

The simplest invariant for distribution germs is given by a pair of integers $(n,p)$ where $n=\dim M$ and $p=\cork \mathcal{D}$. The germs of contact and Engel structures are generic in their respective classes. They also admit local frames generating a finite dimensional Lie algebra structure. The only other distributions that have the same properties are the class of even contact structures and the $1$-dimensional distributions. All of these lie in the range $p(n-p) \le n$. But in the range $p(n-p) > n$, generic distribution germs do not admit local frames which generate finite dimensional Lie algebra, due to the presence of function moduli \cite{montGeneric}.

Contact structures are the simplest kind of \emph{strongly} bracket generating distributions. A distribution $\mathcal{D}$ is called strongly bracket generating if every non-vanishing vector field in $\mathcal{D}$, defined locally around a point $x\in M$, generates the tangent space $T_x M$ by taking Lie bracket with sections of $\mathcal{D}$ only once. Strongly bracket generating distributions are also referred to as \emph{fat distributions} in the literature. In fact, corank-1 fat distributions are the same as the contact ones. Germs of fat distributions in higher corank, are far from being generic \cite{raynerFat}. However, they are interesting in their own right and have been well-studied \cite{geHorizontalCC, montTour}. 

The notion of contact structures can be extended verbatim to complex manifolds (\autoref{exmp:holomorphicContact}). There is also a quaternionic analog of contact structures, which are corank-3  distributions on $(4n+3)$-manifolds (\autoref{exmp:quaternionicContact}). Both these families of distributions satisfy the fatness property. Moreover, these distributions admit local frames generating finite dimensional lie algebras, namely, the complex Heisenberg lie algebra and the quaternionic Heisenberg Lie algebra in place of real Heisenberg algebra for contact structures \cite{montTour}. Hence, they are not generic.\medskip

In this article, we mainly focus on the existence of smooth horizontal immersions in a certain class of corank-$2$ fat distributions, which are referred to here as `degree $2$ fat distributions' (\autoref{defn:degreeOfFatDistribution}) and which include the underlying real bundles of holomorphic contact structures. The main results may be stated as follows.

\begin{theoremX} [\autoref{thm:hPrinHorizImmFatDeg2}, \autoref{thm:existenceHorizImm}] \label{thm:mainTheoremHorizDeg2}
	Let $\mathcal{D}$ be a degree 2 fat distribution on a manifold $M$. Then $\mathcal{D}$-horizontal $\Omega$-regular immersions $\Sigma\to (M,\mathcal{D})$ satisfy the $C^0$-dense $h$-principle provided $\rk \mathcal{D} \ge 4\dim \Sigma$. Furthermore, any map $\Sigma\to M$ can be $C^0$-approximated by an $\Omega$-regular $\mathcal{D}$-horizontal immersion provided $\rk \mathcal{D} \ge \max \{4\dim \Sigma, 5\dim \Sigma - 3\}$.
\end{theoremX}

The existence of a horizontal immersion follows from the $h$-principle by solving the algebraic problem. An important corollary to the above result is the $h$-principle for loops in arbitrary corank $2$ fat distribution on a $6$-dimensional manifold (\autoref{corr:hPrinFat46}). This corollary also follows from \cite{pinoKnotEmbeddingBracketGenerating}.

We would like to emphasize that the $h$-principle obtained in \autoref{thm:mainTheoremHorizDeg2} is in the optimal range (\autoref{rmk:horizontalOptimalRange}). However, this does not rule out horizontal immersions in $(M, \mathcal{D})$ below the critical dimension if $\Omega$-regularity condition is dropped. In \cite{bhowmickFat46}, it has been observed that there are certain corank-$2$ fat distribution germs on $\mathbb{R}^6$ which allow germs of $2$-dimensional submanifolds tangent to the given distribution (\autoref{rmk:existenceGermHorizDeg2}). It may also be pertinent to mention here that holomorphic Legendrian embeddings of an open Riemann surface into $\mathbb{C}^3$, with the standard holomorphic contact structure, satisfy the Oka’s principle \cite{forstnericLegendrianCurves,forstnericHoloLegendrianCurves}.\medskip

More generally, we can consider immersions $u : \Sigma \to M$  which induce a specific distribution $K$ on the domain, i.e., $K = du^{-1}\mathcal{D}$. These are called $K$-isocontact immersions. The $h$-principle in \autoref{thm:mainTheoremHorizDeg2} is derived as a consequence of the following.

\begin{theoremX} [\autoref{thm:hPrinIsocontactDeg2}, \autoref{thm:existenceIsocontact}] \label{thm:mainTheoremIsoContDeg2}
	Let $K$ be a given contact structure on $\Sigma$ and $\mathcal{D}$ be a degree $2$ fat distribution on $M$. Then, $K$-isocontact immersions $\Sigma \to M$ satisfy the $C^0$-dense $h$-principle provided $\rk \mathcal{D} \ge 2\rk K + 4$. Furthermore, if $K,\mathcal{D}$ are cotrivial, then any map $\Sigma \to M$ can be $C^0$-approximated by a $K$-isocontact immersion provided $\rk \mathcal{D}\ge \max \{2 \rk K + 4, 3\rk K - 2\}$.
\end{theoremX}

We also obtain $h$-principles for horizontal and isocontact immersions into quaternionic contact distributions.

\begin{theoremX} [\autoref{thm:hPrinHorizImmQCont}, \autoref{thm:existenceHorizImmQCont}] \label{thm:mainTheoremHorizQuat}
	Let $\mathcal{D}$ be a quaternionic contact structure on $M$. Then, $\mathcal{D}$-horizontal immersions $\Sigma\to M$ satisfy the $C^0$-dense $h$-principle provided $\rk \mathcal{D} \ge 4\dim\Sigma + 4$. Furthermore, any map $\Sigma\to M$ can be $C^0$-approximated by a $\mathcal{D}$-horizontal immersion provided $\rk \mathcal{D} \ge \max\{4 \dim \Sigma + 4, 5\dim\Sigma - 3\}$.
\end{theoremX}

\begin{theoremX} [\autoref{thm:hPrinIsocontactQCont}, \autoref{thm:existenceIsocontactQCont}] \label{thm:mainTheoremIsocontactQCont}
	Let $K$ be a given contact structure on $\Sigma$ and $\mathcal{D}$ be a quaternionic contact structure on $M$. Then, $\Omega$-regular $K$-isocontact immersions $\Sigma \to M$ satisfy the $C^0$-dense $h$-principle provided $\rk \mathcal{D} \ge 4\rk K + 4$. Furthermore, if $K,\mathcal{D}$ are cotrivial, then any map $\Sigma \to M$ can be $C^0$-approximated by an $\Omega$-regular, $K$-isocontact immersion provided $\rk \mathcal{D}\ge \max \{4 \rk K + 4, 6\rk K - 2\}$.
\end{theoremX}

The results follow from the general theory of $h$-principle by applying the sheaf theoretic and analytic techniques. The article is organized as follows. In \autoref{sec:generalHPrinciple}, we recall briefly the language and the main theorems of $h$-principle to keep the article self-contained. Then, in \autoref{sec:revisitKContact} we discuss in detail the $h$-principle of $\Omega$-regular $K$-contact immersions and revisit Gromov's Approximation Theorem for overregular immersions. Next, in \autoref{sec:fatAndDegree} we introduce the notion of `degree' on corank-$2$ fat distributions and study their algebraic properties. Finally, in \autoref{sec:application} we apply the general results of \autoref{sec:revisitKContact} to prove the main theorems, and then discuss some implications of these theorems in symplectic geometry.

\section{Preliminaries of  $h$-principle}
\label{sec:generalHPrinciple}
In this section, we briefly recall certain techniques in the theory of $h$-principle. We refer to \cite{gromovBook} for a detailed discussion of this theory.\medskip

Let $p:X\to V$ be a smooth fibration and $X^{(r)}\to V$ be the $r$-jet bundle associated with $p$. The space $\Gamma X$ consisting of smooth sections of $X$ has the $C^\infty$-compact open topology, whereas $\Gamma X^{(r)}$ has the $C^0$-compact open topology. Any differential condition on sections of the fibration defines a subset in the jet space $X^{(r)}$, for some integer $r \ge 0$. Hence, in the language of $h$-principle, a \emph{differential relation} is by definition a subset $\mathcal{R}\subset X^{(r)}$, for some $r\geq 0$. A section $x$ of $X$ is said to be a \emph{solution} of the differential relation $\mathcal{R}$ if its $r$-jet prolongation $j^r_x : V \to X^{(r)}$ maps $V$ into $\mathcal{R}$. Let $\Sol \mathcal{R}$ denote the space of smooth solutions of $\mathcal{R}$ and let $\Gamma \mathcal{R}$ denote the space of sections of the jet bundle $X^{(r)}$ having their images in $\mathcal{R}$. The $r$-jet map then takes $\Sol \mathcal{R}$ into $\Gamma \mathcal{R}$; in fact, this is an injective map, so that $\Sol \mathcal{R}$ may be viewed as a subset of $\Gamma \mathcal{R}$. Any section in the image of this map is called a \emph{holonomic} section of $\mathcal{R}$.

\begin{defn}
	If every section of $\mathcal{R}$ can be homotoped to a solution of $\mathcal{R}$ then we say that $\mathcal{R}$ satisfies the \emph{ordinary $h$-principle} (or simply, \emph{$h$-principle}).
\end{defn}

\begin{defn}
	We say that $\mathcal{R}$ satisfies the \emph{parametric $h$-principle} if $j^r:\Sol \mathcal{R}\to \Gamma \mathcal{R}$ is a weak homotopy equivalence; this means that the solution space of $\mathcal{R}$ is classified by the space $\Gamma \mathcal{R}$. $\mathcal{R}$ satisfies the \emph{local} (parametric) $h$-principle if $j^r$ is a local weak homotopy equivalence.
\end{defn}

\begin{defn}
	$\mathcal{R}$ is said to satisfy the \emph{$C^0$-dense} $h$-principle if for every $F_0 \in \Gamma\mathcal{R}$ with base map $f_0 = \bs F_0$ and for any neighborhood $U$ of $\im f_0$ in $X$, there exists a homotopy $F_t \in \Gamma\mathcal{R}$ joining $F_0$ to a holonomic $F_1 = j^r_{f_1}$ such that the base map $f_t = \bs F_t$ satisfies $\im f_t \subset U$ for all $t\in [0,1]$.
\end{defn}

We shall now state the main results of the sheaf technique and analytic technique, the combination of which gives the global $h$-principle for many interesting relations, including closed relations arising from partial differential equations.

\subsection{Sheaf Technique in $h$-Principle}\label{sec:generalHPrinciple:sheaf}
We begin with some terminology of topological sheaves $\Phi$ on a manifold $V$. For any arbitrary set $C\subset V$, we denote by $\Phi(C)$ the collection of sections of $\Phi$ defined on some arbitrary open neighborhood $\Op C$.
\begin{defn}\label{defn:sheafMicroFlexible}
	A topological sheaf $\Phi$ is called \emph{flexible} (resp. \emph{microflexible}) if for every pair of compact sets $A\subset B\subset V$, the restriction map $\rho_{B,A}:\Phi(B)\to\Phi(A)$ is a Serre fibration (resp. microfibration). Recall that $\rho_{B,A}$ is a  microfibration if every homotopy lifting problem $(F,\tilde{F}_0)$, where $F:P\times I\to \Phi(A)$ and $\tilde{F}_0:P\to \Phi(B)$ are (quasi)continuous maps, admits a partial lift $\tilde{F}:P\times [0,\varepsilon]\to \Phi(B)$ for some $\varepsilon>0$.
\end{defn}

\begin{defn}
	Given two sheaves $\Phi,\Psi$ on $V$, a sheaf morphism $\alpha : \Phi \to \Psi$ is called a \emph{weak homotopy equivalence} if, for each open $U\subset V$, $\alpha(U):\Phi(U)\to \Psi(U)$ is a weak homotopy equivalence. The map $\alpha$ is a \emph{local} weak homotopy equivalence if, for each $v\in V$, the induced map $\alpha_v : \Phi(v) \to \Psi(v)$ on the stalk is a weak homotopy equivalence.
\end{defn}

We now quote a general result from the theory of topological sheaves.

\begin{theorem}[Sheaf Homomorphism Theorem]\label{thm:sheafHomoTheorem} \cite[pg. 77]{gromovBook}
	Every local weak homotopy equivalence $\alpha:\Phi\to\Psi$ between flexible sheaves $\Phi,\Psi$ is a weak homotopy equivalence.
\end{theorem}

Now, suppose $\Phi$ is the sheaf of solutions of a relation $\mathcal{R}\subset X^{(r)}$, and $\Psi$ is the sheaf of sections of $\mathcal{R}$. Then we have the obvious sheaf homomorphism given by the $r$-jet map, $J = j^r :\Phi \to \Psi$. In this case, the sheaf $\Psi$ is always flexible. Hence, if $\Phi$ is flexible and $J$ is a local weak homotopy equivalence, then the relation $\mathcal{R}$ satisfies the parametric $h$-principle. But in general, $\Phi$ fails to be flexible, though the solution sheaves for many relations do satisfy the micro-flexibility property. The following theorem gives a sufficient condition for the flexibility of a solution sheaf when restricted to a submanifold of positive codimension. For any manifold $V$, let $\diff(V)$ denote the pseudogroup of local diffeomorphisms of $V$.

\begin{theorem}[Flexibility Theorem]\label{thm:mainFlexibilityTheorem} \cite[pg. 78]{gromovBook}
	Let $\Phi$ be a microflexible sheaf and $V_0 \subset V$ be a submanifold of positive codimension. If $\Phi$ is invariant under the action of certain subset of $\diff(V)$ which sharply moves $V_0$, then the restriction sheaf $\Phi|_{V_0}$ is flexible. (That is, for any compact sets $A, B \subset V_0$ with $A\subset B$, the restriction map $\rho_{B,A}$ is a fibration.)
\end{theorem}

We refer to \cite[pg. 82]{gromovBook} for the definition of sharply moving diffeotopies and also to \cite[pg. 139]{eliashbergBook} for the related notion of capacious subgroups.
\begin{example}\label{exmp:sharplyMovingDiffeo}
	We mention two important classes of sharply moving diffeotopies here which will be of interest to us.
	\begin{enumerate}
		\item\label{exmp:sharplyMovingDiffeo:1} If $V= V_0\times \mathbb{R}$, then we can identify a subpseudogroup $\diff(V,\pi)\subset\diff(V)$ consisting of diffeomorphisms $\phi : V \to V$ such that $\pi \circ \phi = \phi$. We shall refer to them as fiber preserving diffeomorphisms. It follows that $\diff(V,\pi)$ sharply moves $V_0$ in $V$.
		
		\item\label{exmp:sharplyMovingDiffeo:2} Let $K$ be a contact structure on $V$. Then, the collection of contact diffeotopies of $V$ sharply moves any submanifold $V_0 \subset V$ \cite[pg. 339]{gromovBook}.
	\end{enumerate}
\end{example}
As a consequence of the above theorem, we get the following result.
\begin{theorem} \label{thm:openManifoldHPrin}
	Let $V_0 \subset V$ be a submanifold positive codimension. A relation $\mathcal{R}$ satisfies the parametric $h$-principle near $V_0$ provided the following conditions hold:
	\begin{enumerate}
		\item $\mathcal{R}$ satisfies the local $h$-principle, and
		\item the solution sheaf of $\mathcal{R}$ satisfies the hypothesis of \autoref{thm:mainFlexibilityTheorem}.
	\end{enumerate} 
\end{theorem}

It can be easily seen that any \emph{open} relation satisfies the local $h$-principle and its solution sheaf is microflexible. A large class of \emph{non-open} relations also enjoy the same properties, as we shall discuss below.\smallskip

Let $J^r(V,M)$ denote the $r$-jet space associated with $C^r$ maps from a manifold $V$ to $M$. To study the $h$-principle for a relation $\mathcal{R} \subset J^r(V,M)$ on an arbitrary manifold $V$, the general idea is to first embed $V$ in a higher dimensional manifold $\tilde{V}$ and to find a relation $\tilde{\mathcal{R}}$ on $\tilde{V}$ which is an extension of $\mathcal{R}$ in the sense explained below. 

If $V$ is a submanifold of $\tilde{V}$, then there is a canonical restriction  morphism $C^\infty(\tilde{V},M) \to C^\infty(V, M)$ which then induces a map $\rho:J^r(\tilde{V},M)|_V \to J^r(V,M)$.

\begin{defn}\label{defn:microextension}
	A relation $\tilde{\mathcal{R}}$ on $\tilde{V}$ will be called an \emph{extension} of $\mathcal{R}\subset J^r(V,M)$ if $\rho$ maps $\tilde{\mathcal{R}}|_V$ into $\mathcal{R}$. An extension $\tilde{\mathcal{R}}$ will be called a \emph{microextension} of $\mathcal{R}$ if the induced maps $\rho_* : \Gamma(\tilde{\mathcal{R}}|_O) \to \Gamma(\mathcal{R}|_O)$ are surjective for contractible open sets $O \subset V$.
\end{defn}

It is to be noted that the notion of microextension as in \cite[pg. 85]{gromovBook} is different from the notion defined above. 

\subsection{Analytic Technique in $h$-Principle}\label{sec:generalHPrinciple:analytic}
Suppose $X\to V$ is a fibration and $G\to V$ is a vector bundle. Let us consider a $C^\infty$-differential operator $\mathfrak{D}: \Gamma X\to \Gamma G$ of order $r$, given by the $C^\infty$-bundle map $\Delta : X^{(r)}\to G$, known as the \emph{symbol} of the operator, satisfying $$\Delta\circ j^r_x = \mathfrak{D}(x), \quad \text{for $x \in \Gamma X$.}$$
Suppose that $\mathfrak{D}$ is \emph{infinitesimally invertible} over a subset $\mathcal{S} \subset \Gamma X$, where $\mathcal{S}$ consists of all $C^\infty$-solutions of a $d$-th order \emph{open} relation $S\subset X^{(d)}$, for some $d\ge r$. Roughly speaking, this means that there exists an integer $s\geq 0$ such that for each $x\in \mathcal{S}$, the linearization of $\mathfrak{D}$ at $x$ admits a right inverse, which is a linear differential operator of order $s$. The integer $s$ is called the \emph{order} of the inversion, while $d$ is called the \emph{defect}. The elements of $\mathcal{S}$ are referred to as \emph{$S$-regular} (or simply, \emph{regular}) maps.

It follows from the Nash-Gromov Implicit Function Theorem \cite[pg. 117]{gromovBook} for smooth differential operators that, $\mathfrak{D}$ restricted to $\mathcal{S}$ is an open map with respect to the fine $C^\infty$-topologies if the operator is infinitesimally invertible on $\mathcal{S}$. In particular, it implies that $\mathfrak{D}$ is locally invertible at $S$-regular maps. Explicitly, if $x_0\in\mathcal{S}$ and $\mathfrak{D}(x_0)=g_0$, then there exists a neighborhood $\mathcal{V}_0$ of the zero section in $\Gamma G$ and an operator $\mathfrak{D}_{x_0}^{-1}:\mathcal{V}_0\to \mathcal{S}$ such that for all $g\in\mathcal{V}_0$ we have $\mathfrak{D}(\mathfrak{D}_{x_0}^{-1}(g))=g_0+g$. We shall call $\mathfrak{D}_{x_0}^{-1}$ a local inverse of $\mathfrak{D}$ at $x_0$.

\begin{defn} 
	Fix some $g\in\Gamma G$. A germ $x_0\in S$ at a point $v \in V$ is called an \emph{infinitesimal solution} of $\mathfrak{D}(x)=g$ of order $\alpha$ if $j^\alpha_{\mathfrak{D}(x_0) - g}(v) = 0$.
\end{defn}

Let $\mathcal{R}^\alpha(\mathfrak{D}, g)\subset X^{r+\alpha}$ denote the relation consisting of jets represented by \emph{infinitesimal solutions} of $\mathfrak{D}(x)=g$ of order $\alpha$, at points of $V$. For $\alpha\geq d-r$, define the relations $\mathcal{R}_\alpha$ as follows:
$$\mathcal{R}_\alpha = \mathcal{R}_\alpha(\mathfrak{D},g,S) := \mathcal{R}^{\alpha}\cap (p_d^{r + \alpha})^{-1}S,$$
where $p^{r+\alpha}_d:X^{(r+\alpha)}\to X^{(d)}$ is the canonical projection of the jet spaces. Then, for all $\alpha\ge d-r$, the relations $\mathcal{R}_\alpha$ have the same set of $C^\infty$-solutions, namely, the $S$-regular $C^\infty$-solutions of $\mathfrak{D}(x)=g$. Denote the sheaf of solutions of any such $\mathcal{R}_\alpha$ by $\Phi$, and let $\Psi_\alpha$ denote the sheaf of sections of $\mathcal{R}_\alpha$.
\begin{theorem}\cite[pg. 119-120]{gromovBook}\label{thm:microflexibleLocalWHESheafTheorem}
	Suppose $\mathfrak{D}$ is a smooth differential operator of order $r$, which admits an infinitesimal inversion of order $s$ and defect $d$ on an open subset $S \subset X^{(d)}$, where $d \ge r$. Then for $\alpha\ge \max\{d+s, 2r+2s\}$ the jet map $j^{r+\alpha} : \Phi\to\Psi_\alpha$ is a local weak homotopy equivalence. Also, $\Phi$ is a microflexible sheaf. 
\end{theorem}

We end this section with a theorem on the Cauchy initial value problem associated with the equation $\mathfrak{D}(x) = g$.

\begin{theorem}\cite[pg. $144$]{gromovBook}  \label{thm:consistenInversion}
	Suppose $\mathfrak{D}$ is a differential operator of order $r$, admitting an infinitesimal inversion of order $s$ and defect $d$ over $\mathcal{S}$. Let $x_0\in \mathcal{S}$ and $g_0 = \mathfrak{D}(x_0)$. Suppose $V_0\subset V$ is a codimension $1$ submanifold without boundary and $g\in \Gamma G$ satisfies 
	$$j^{l}_g|_{V_0} = j^{l}_{g_0}|_{V_0} \text{ for some }l \ge 2r + 3s + \max\{d, 2r + s\}.$$
	Then, there exists an $x\in\mathcal{S}$ such that $\mathfrak{D}(x)= g$ on ${\Op V_0}$ and $$j^{2r+s-1}_x|_{V_0} = j^{2r+s-1}_{x_0}|_{V_0}.$$
\end{theorem}
The above result follows from a stronger version of the Implicit Function Theorem.

\section{Revisiting the $h$-Principle of Regular $K$-Contact Immersions} \label{sec:revisitKContact}
	Throughout this section $\mathcal{D}$ will denote an arbitrary corank-$p$ distribution on a manifold $M$ and $\lambda:TM\to TM/\mathcal{D}$ will denote the quotient map. For every pair of local sections $X,Y$ in $\mathcal{D}$, $\lambda([X,Y])$ is a local section of the bundle $TM/\mathcal{D}$. The map
\begin{align*}
	\Gamma(\mathcal{D})\times\Gamma(\mathcal{D}) &\to \Gamma(TM/\mathcal{D})\\
	(X,Y) &\mapsto  -\lambda([X,Y])
\end{align*}
is $C^\infty(M)$-linear and hence induces a bundle map $\Omega:\Lambda^2\mathcal{D} \to TM/\mathcal{D}$, which is called the \emph{curvature form} of the distribution $\mathcal{D}$. Any  local trivialization of the bundle $TM/\mathcal{D}$ defines local 1-forms $\lambda^i$, $i=1,\dots,p$, such that $\mathcal{D}\underset{loc.}{=}\cap_{i=1}^p\ker\lambda^i$. Then $\Omega$ can be locally expressed as follows:
$$\Omega \underset{loc.}{=} \big(d\lambda^1|_\mathcal{D}, \,\ldots,\, d\lambda^p|_\mathcal{D}\big).$$
The span $\langle d\lambda^1|_\mathcal{D},\ldots,d\lambda^p|_\mathcal{D}\rangle$ is clearly independent of the choice of defining $1$-forms $\lambda^1,\ldots,\lambda^p$ for $\mathcal{D}$.

\begin{remark}\label{rmk:curvatureConnection}
	The quotient map $\lambda$ can be treated as a $TM/\mathcal{D}$-valued $1$-form on $M$. If $\nabla$ is an arbitrary connection on the quotient bundle $TM/\mathcal{D}$, then the curvature form $\Omega$ can be given as $\Omega = d_\nabla \lambda|_\mathcal{D}$.
\end{remark}

\begin{defn}
	A smooth map $u:\Sigma\to M$ is \emph{$\mathcal{D}$-horizontal} if the differential $du$ maps $T\Sigma$ into $\mathcal{D}$.
\end{defn}

\begin{defn}\cite[pg. 338]{gromovBook}\label{defn:contactMap}
	Given a subbundle $K\subset T\Sigma$, we say a map $u:\Sigma\to (M,\mathcal{D})$ is \emph{$K$-contact} if $$du(K_\sigma)\subset T_{u(\sigma)}\mathcal{D},\quad\text{for each $\sigma\in\Sigma$.}$$
	A $K$-contact map $u:(\Sigma,K)\to (M,\mathcal{D})$ is called $K$-\emph{isocontact} (or, simply \emph{isocontact}) if we have $K = du^{-1}(\mathcal{D})$.
\end{defn}

In what follows below, $\Sigma$ will denote an arbitrary manifold and $K$ will denote an arbitrary but fixed subbundle of $T\Sigma$, unless mentioned otherwise. For any contact map $u:(\Sigma,K)\to (M,\mathcal{D})$, we have an induced bundle map
\begin{align*}
	\tilde{du} : T\Sigma/K &\longrightarrow u^*TM/\mathcal{D}\\
	X \mod K &\longmapsto du(X) \mod\mathcal{D}
\end{align*}
Clearly, a contact \emph{immersion} $u:(\Sigma,K)\to (M,\mathcal{D})$ is isocontact if and only if $\tilde{du}$ is a monomorphism. Hence, for an isocontact immersion $(\Sigma,K)\to (M,\mathcal{D})$ to exist, the following numerical constraints must necessarily be satisfied: $$\rk K \le \rk \mathcal{D} \quad\text{and}\quad \cork K \le \cork\mathcal{D}.$$
$K$-contactness automatically imposes a differential condition involving the curvatures of the two distributions.

\begin{prop}\label{prop:isocontactCurvatureCondition}
	If $u : (\Sigma,K) \to (M,\mathcal{D})$ is a $K$-contact map, then 
	\begin{equation}\label{eqn:curvatureEquation}
		u^*\Omega_\mathcal{D}|_K = \tilde{du} \circ \Omega_K,
	\end{equation}
	where $\Omega_K,\Omega_\mathcal{D}$ are the curvature forms of $K$ and $\mathcal{D}$ respectively. Equivalently, we have the following commutative diagram
	\[\begin{tikzcd}
		\Lambda^2K \arrow{d}[swap]{\Omega_K} \arrow{r}{du} & \Lambda^2 \mathcal{D}  \arrow{d}{\Omega_\mathcal{D}}\\
		T\Sigma/K \arrow{r}[swap]{\tilde{du}} & TM/\mathcal{D}
	\end{tikzcd}\]
\end{prop}
If $K=T\Sigma$, then $\Omega_K = \Omega_{T\Sigma} = 0$. Hence, for a horizontal immersion $u:\Sigma\to M$ this gives the \emph{isotropy} condition, namely, $u^*\Omega_\mathcal{D} = 0$.\medskip

For simplicity, we assume that $\mathcal{D}$ is globally defined as the common kernel of $\lambda^1,\ldots,\lambda^p$, and consider the differential operator
\begin{align*}
	\opCont : C^\infty(\Sigma,M) &\to \Gamma \hom(K,\mathbb{R}^p) = \Omega^1(K,\mathbb{R}^p)\\
	u &\mapsto \big(u^*\lambda^s|_K\big)_{s=1}^p.
\end{align*}
Clearly, $K$-contact maps are solutions of $\opCont(u) = 0$. Recall that the tangent space of $C^\infty(\Sigma,M)$ at some $u:\Sigma\to M$ can be identified with the space of vector fields of $M$ \emph{along the map $u$}, i.e, the space of sections of $u^*TM$. Any such vector field $\xi \in \Gamma u^*TM$ can be represented by a family of maps $u_t : \Sigma\to M$ such that $u_0 = u$ and $\xi_\sigma = \frac{d}{dt}\big|_{t=0} u_t(\sigma)$ for $\sigma \in\Sigma$. Then, the linearization of $\opCont$ at $u$ is given by $$\linCont_u(\xi) = \frac{d}{dt}\Big|_{t=0} \opCont(u_t).$$ 
By the Cartan formula we get
\begin{align*}
	\linCont_u : \Gamma u^*TM &\to \Gamma \hom(K, \mathbb{R}^p)\\
	\xi &\mapsto \Big(\iota_\xi d\lambda^s + d\big(\iota_\xi\lambda^s\big)\Big)\Big|_K.
\end{align*}
Restricting $\linCont_u : \Gamma u^*TM\to \Gamma\hom(K,\mathbb{R}^p)$ to the subspace $\Gamma u^*\mathcal{D}$ we get a $C^\infty(M)$-linear operator
\begin{align*}
	\resLinCont_u : \Gamma u^*\mathcal{D}&\to \Gamma\hom(K,\mathbb{R}^p)\\
	\xi &\mapsto \Big(\iota_\xi d\lambda^s\Big)\Big|_K. 
\end{align*}
The associated bundle map will also be denoted by the same symbol.
\begin{defn}\label{defn:contOpenRegular}
	A smooth immersion $u :\Sigma\to M$ is said to be \emph{$(d\lambda^s)$-regular} if $\resLinCont_u$ is an epimorphism.
\end{defn}

We shall denote the space of all $(d\lambda^s)$-regular immersions by $\mathcal{S}$. Such maps $u$ are solutions to a first order \emph{open} relation $S\subset J^1(\Sigma,M)$.\smallskip

\begin{prop}\label{prop:opContInfinitesimallyInvertible}
	$\opCont$ is a first order differential operator having an infinitesimal inversion over $\mathcal{S}$ of order $s=0$, and defect $d=1$.
\end{prop}
\begin{proof}
	Since $\resLinCont_u$ is a surjective vector bundle map, we can always choose a right inverse $\mathfrak{M}_u : \Gamma (K, \mathbb{R}^p) \to \Gamma u^*\mathcal{D}$. By a choice of some auxiliary Riemannian metric on $M$, we can make sure $\mathfrak{M}_u$ depends smoothly on $u$. Clearly $\linCont_u \circ \mathfrak{M}_u = \resLinCont_u \circ \mathfrak{M}_u = \Id$, and hence $\linCont_u$ has a right inverse of order $s = 0$. Since $\mathcal{S}$ is the solution of a first order relation $S$, the inversion has defect $d = 1$.
\end{proof}

In general, $(d\lambda^s)$-regularity depends on the choice of $\lambda^s$, but the space of $(d\lambda^s)$-regular, $K$-contact immersions $(\Sigma,K)\to (M,\mathcal{D})$ is independent of any such choice. Indeed, if $du(K)\subset\mathcal{D}$, then
$$\resLinCont_u(\xi) = \iota_\xi\Omega \big|_K, \quad \text{for $\xi \in \Gamma u^*\mathcal{D}$,}$$ where $\Omega$ is the curvature $2$-form of $\mathcal{D}$.

\begin{remark}
	For a general distribution $\mathcal{D}$, not necessarily cotrivializable, we look at the operator $$\opCont : u \mapsto u^*\lambda|_K \in \Gamma\hom(K, u^*TM/\mathcal{D}), \quad \text{for any $u:\Sigma\to M$}.$$
	To put this in a rigorous framework, consider the infinite dimensional space $\mathcal{B} = C^\infty(\Sigma, M)$ and then consider the infinite dimensional vector bundle $\mathcal{E}\to \mathcal{B}$ with fibers $\mathcal{E}_u = \Gamma \hom(K, u^*TM/\mathcal{D})$. Then, $\opCont$ can be seen as a section of this vector bundle. To identify the linearization operator, we choose any connection $\nabla$ on $TM/\mathcal{D}$, which in turn induces a parallel transport on $\mathcal{E}$. We then get $\linCont_u (\xi) = \big(\iota_\xi d_\nabla \lambda + d_\nabla \iota_\xi \lambda\big)\big|_K$ for $\xi \in\Gamma u^*TM$. Restricting $\linCont_u$ to $\Gamma u^*\mathcal{D}$, we get the $C^\infty(\Sigma)$-linear map $$\resLinCont_u : \xi \mapsto \iota_\xi d_\nabla \lambda|_K, \quad \xi \in \Gamma u^*\mathcal{D}.$$
	In view of \autoref{rmk:curvatureConnection}, $\resLinCont_u(\xi) = \iota_\xi\Omega|_K$ for a $K$-contact immersion $u : \Sigma \to M$, which matches with our earlier description.
\end{remark}

\begin{defn}\label{defn:contOmegaRegular}
	A subspace $V\subset \mathcal{D}_y$ is called \emph{$\Omega$-regular} if the map
	\begin{equation}\label{eqn:omegaRegular}
		\begin{aligned}
			\mathcal{D}_y &\to \hom(V, TM/\mathcal{D}|_y)\\
			\xi &\mapsto \iota_\xi \Omega|_V
		\end{aligned}
	\end{equation}
	is surjective. A $K$-contact immersion $u:(\Sigma,K)\to (M,\mathcal{D})$ is called $\Omega$-\emph{regular} if $du_x(K_x) \subset \mathcal{D}_{u(x)}$ is $\Omega$-regular for every $x\in \Sigma$, equivalently, if $\resLinCont_u$ is a bundle epimorphism.
\end{defn}

	In order to study the $K$-contact immersions, let $\relCont_\alpha =$ $\relCont_\alpha(\opCont, 0, S) \subset J^{\alpha+1}(\Sigma,M)$ be the relation consisting of $(d\lambda^s)$-regular infinitesimal solutions of $\opCont = 0$ of order $\alpha$.  Then $\relCont_\alpha$, for all $\alpha \ge d -r = 0$, have the same $C^\infty$-solutions spaces, namely the $\Omega$-regular $K$-contact immersions. We introduce the following notation for the solution sheaf and the sheaf of sections of $\relCont_\alpha$: $$\solCont = \Sol \relCont_\alpha, \quad \secCont_\alpha = \Gamma \relCont_\alpha.$$

\begin{obs}\label{obs:relContMicroflexibleLocalWHE}
	From \autoref{thm:microflexibleLocalWHESheafTheorem} and \autoref{prop:opContInfinitesimallyInvertible} we obtain that
	\begin{itemize}
		\item $\solCont$ is microflexible, and
		\item for $\alpha \ge \max\{d + s, 2r + 2s\} = 2$, $\relCont_\alpha$ satisfies the parametric \emph{local} $h$-principle, i.e, the jet map $j^{\alpha + 1} : \solCont \to \secCont_\alpha$ is a \emph{local} weak homotopy equivalence.
	\end{itemize}
\end{obs}

In general, there is no natural $\diff(\Sigma)$ action on $\solCont$. However, when $K = T\Sigma$ then it is the sheaf of horizontal immersions for which we have the following results.

\begin{theorem}\cite{gromovBook}\label{thm:hPrinHorizOpenAlpha}
	If $\Sigma$ is an open manifold, then the relation $\relHor_\alpha$ satisfies the parametric $h$-principle for $\alpha \ge 2$.
\end{theorem}
\begin{proof}
	We observe that the natural $\diff(\Sigma)$-action on $C^\infty(\Sigma, M)$ preserves $\mathcal{D}$-horizontality and $\Omega$-regularity. Hence, $\diff(\Sigma)$ acts on $\solHor = \Sol\relHor_\alpha$ for $\alpha \ge 0$. Then a direct application of \autoref{thm:openManifoldHPrin} gives us that $j^{\alpha + 1} : \solHor \to \Gamma \relHor_\alpha$ is a weak homotopy equivalence for $\alpha \ge 2$. In other words, $\relHor_\alpha$ satisfies the parametric $h$-principle for $\alpha \ge 2$.
\end{proof}

\begin{theorem}\label{thm:hPrinContOpenAlpha}
	Let $K$ be a contact structure on $\Sigma$. Then the relation $\relCont_\alpha$ satisfies the parametric $h$-principle for $\alpha \ge 2$ near any positive codimensional submanifold $V_0 \subset \Sigma$.
\end{theorem}
\begin{proof}
	Since the group of contact diffeomorphisms sharply moves any submanifold of $\Sigma$ (\autoref{exmp:sharplyMovingDiffeo}), for any submanifold $V_0 \subset \Sigma$ of positive codimension, we have the $h$-principle via an application of \autoref{thm:mainFlexibilityTheorem}.
\end{proof}

\subsection{The Relation $\relCont$}
We now define a first order relation, taking into account the curvature condition (\autoref{eqn:curvatureEquation}). This relation will also have the same $C^\infty$-solution sheaf $\solCont$.

\begin{defn}\label{defn:relCont}
	Let $\relCont\subset J^1(\Sigma,M)$ denote the relation consisting of $1$-jets $(x,y, F:T_x\Sigma\to T_y M)$ satisfying the following:
	\begin{enumerate}
		\item \label{defn:relCont:1} $F$ is injective and $F(K_x)\subset\mathcal{D}_y$.
		
		\item \label{defn:relCont:2} $F$ is $\Omega$-regular, i.e, the linear map
		\begin{align*}
			\mathcal{D}_y &\to \hom(K_x, TM/\mathcal{D}|_y)\\
			\xi &\mapsto F^*(\iota_\xi\Omega)|_K = \big(X\mapsto \Omega(\xi, FX)\big) 
		\end{align*} is surjective (\autoref{eqn:omegaRegular}).
	
		\item \label{defn:relCont:3} $F$ abides by the \emph{curvature condition}, $F^*\Omega|_{K_x} = \tilde{F}\circ \Omega_K|_x$, where $\tilde F : T\Sigma/K|_x \to TM/\mathcal{D}|_y$ is the morphism induced by $F$ (\autoref{eqn:curvatureEquation}).
	\end{enumerate}
	The subrelation $\relICont\subset\relCont$ further satisfies the condition that
	\begin{itemize}
		\item[(4)] \label{defn:relCont:4} $\tilde{F}$ is injective.
	\end{itemize}
	If $K = T\Sigma$, we shall denote the corresponding relation by $\relHor$, whose solution space consists of $\Omega$-regular horizontal immersions.
\end{defn}

It is immediate from the definition that $\solCont = \Sol \relCont$. We shall refer to a section of $\relCont$ as a \emph{formal} $\Omega$-regular, $K$-contact immersion $(\Sigma,K)\to (M,\mathcal{D})$. We have the following result, which will be needed later in the proof of \autoref{prop:relContExtensionHPrin}.

\begin{lemma} \label{lemma:relContRetraction}
	The following holds true for the relation $\relCont$.
	\begin{enumerate}
		\item \label{lemma:relContRetraction:1} For each $(x,y)\in \Sigma\times M$, the subset $\relCont_{(x,y)}$ is a submanifold of $J^1_{(x,y)}(\Sigma,M)$.
		\item \label{lemma:relContRetraction:2} $\relCont$ is a submanifold of $J^1(\Sigma,M)$.
		\item \label{lemma:relContRetraction:3} The projection map $p = p^1_0:J^1(\Sigma,M)\to J^0(\Sigma,M)$ restricts to a submersion on $\relCont$.
	\end{enumerate}
\end{lemma}
\begin{proof}
	Note that $J^1(\Sigma,M)$ and $\hom(K,TM/\mathcal{D})$ are both vector bundles over $J^0(\Sigma,M)=\Sigma\times M$. Consider the bundle map
	\[\begin{tikzcd}
		\Xi_1 : J^1(\Sigma,M) \arrow{rd} \arrow{rr} &&\hom(K, TM/\mathcal{D}) \arrow{ld} \\
		&J^0(\Sigma,M)
	\end{tikzcd}\]
	defined over $J^0(\Sigma,M) = \Sigma\times M$ by
	\begin{align*}
		\Xi_1|_{(x,y)} : J^1_{(x,y)}(\Sigma,M) &\to \hom(K_x, TM/\mathcal{D}|_y)\\
		\big(x,y, F\big) &\mapsto F^*\lambda|_{K_x} = \lambda\circ F|_{K_x}
	\end{align*}
	Since $\lambda$ is an epimorphism, it is immediate that $\Xi_1$ is a bundle epimorphism and $\ker \Xi_1$ is a vector bundle over $J^0(\Sigma,M)$ given as
	$$\ker \Xi_1|_{(x,y)} = \big\{ (x,y, F) \;\big|\; F(K_x) \subset \mathcal{D}_y \big\}.$$
	Next, consider a fiber-preserving map $\Xi_2 : \ker\Xi_1 \to \hom(\Lambda^2 K, TM/\mathcal{D})$ over $J^0(\Sigma,M)$ given by
	\begin{align*}
		\Xi_2|_{(x,y)} : \ker\Xi_1|_{(x,y)} &\to \hom\big(\Lambda^2 K_x, TM/\mathcal{D}|_y\big)\\
		F &\mapsto F^*\Omega|_{K_x} - \tilde{F} \circ \Omega_{K_x} := \Big(X\wedge Y \mapsto \Omega(FX, FY) - \tilde{F}\circ\Omega_{K_x}(X,Y)\Big)
	\end{align*}
	where $\tilde{F} : T\Sigma/K|_x \to TM/\mathcal{D}|_y$ is the induced map and $\Omega_K:\Lambda^2 K \to T\Sigma/K$ is the curvature $2$-form of $K$. Let $\mathcal{R}_\Omega\subset J^1(\Sigma,M)$ be the space of jets satisfying (\ref{defn:relCont:1}) and (\ref{defn:relCont:2}) of \autoref{defn:relCont}. We note that
	$$\relCont_{(x,y)} = \Xi_2|_{(x,y)}^{-1}(0) \cap \underbrace{\{\text{$\Omega$-regular injective linear maps $T_x\Sigma\to T_yM$, mapping $K_x$ into $\mathcal{D}_y$ }\}}_{\mathcal{R}_\Omega|_{(x,y)} }.$$
	We can verify that $\mathcal{R}_\Omega|_{(x,y)}$ consists of regular points of $\Xi_2|_{(x,y)}$. Consequently, $\relCont_{(x,y)}$ is a submanifold of $J^1_{(x,y)}(\Sigma,M)$.
	Now, since $\Xi_2 : \ker\Xi_1 \to \hom(\Lambda^2 K, TM/\mathcal{D})$ is a fiber-preserving map, it follows that it is regular at all points of $\mathcal{R}_\Omega$ and therefore, $$\relCont = \Xi_2^{-1} \big(\textbf{0}\big) \cap \mathcal{R}_\Omega$$
	is a submanifold of $J^1(\Sigma,M)$. Here $\textbf{0} = \textbf{0}_{\Sigma\times M}\hookrightarrow \hom(\Lambda^2 K, TM/\mathcal{D})$ is the $0$-section.
	Lastly, we consider the commutative diagram
	\[\begin{tikzcd}
		\mathcal{R}_\Omega \subset \ker\Xi_1 \arrow{rr}{\Xi_2} \arrow{rd}[swap]{p^1_0|_{\relCont}} &&\hom(\Lambda^2 K, TM/\mathcal{D}) \arrow{ld}{\pi}\\
		&J^0(\Sigma,M)
	\end{tikzcd}\]
	Since $\Xi_2$ is a submersion on $\mathcal{R}_\Omega$, $p^1_0|_{\relCont}$ is also a submersion.
\end{proof}

We end this section with the following lemma which relates $\relCont_\alpha$ with $\relCont$ for $\alpha\geq 1$.
\begin{lemma}\label{lemma:jetLiftingIsoCont}\label{LEMMA:JETLIFTINGISOCONT}
	For any $\alpha\ge 1$, the jet projection map $p=p^{\alpha+1}_1:J^{\alpha+1}(\Sigma,M)\to J^1(\Sigma,M)$ maps the relation $\relCont_\alpha$ surjectively onto $\relCont$. Furthermore, for each $(x,y)\in\Sigma\times M$, the map $p:\relCont_\alpha|_{(x,y)}\to \relCont|_{(x,y)}$ has contractible fibers. Moreover, any section of $\relCont$ defined over a contractible chart in $\Sigma$ can be lifted to $\relCont_\alpha$ along $p$.  
\end{lemma}

We postpone the proof of the above lemma to \autoref{sec:jetLiftingLemmaProof}. We get the following from \autoref{obs:relContMicroflexibleLocalWHE}.
\begin{corr}\label{corr:relContLocalWHE}
	The induced sheaf map $j^1 : \Sol \relCont \to \Gamma\relCont$ is a local weak homotopy equivalence.
\end{corr}
\begin{proof}
	By an argument presented in \cite[pg. 77-78]{gromovBook}, \autoref{lemma:jetLiftingIsoCont} implies that the sheaf map $p : \Gamma\relCont_\alpha\to \Gamma \relCont$ is a \emph{local} weak homotopy equivalence. Then, because of \autoref{obs:relContMicroflexibleLocalWHE}, $j^1 : \Sol \relCont \to \Gamma \relCont$ is a \emph{local} weak homotopy equivalence.
\end{proof}
Thus, the relations $\relCont$ (and hence $\relHor$), satisfy the \emph{local} parametric $h$-principle. We have the following corollary to \autoref{thm:hPrinHorizOpenAlpha}.
\begin{corr}\label{corr:hPrinHorizImmOpen}
	If $\Sigma$ is an open manifold, then the relation $\relHor$ satisfies the parametric $h$-principle.
\end{corr}

\subsection{$h$-principle for $\relCont$}
In order to get the $h$-principle for $\relCont$ on an arbitrary manifold $\Sigma$, the general plan is to embed $\Sigma$ in a manifold $\tilde{\Sigma}$ with a distribution $\tilde{K}$ such that $\tilde{K}|_\Sigma \cap T\Sigma =K$. We define an operator $\opContTilde$ for the pair $(\tilde{\Sigma},\tilde{K})$ as we did in the case of $(\Sigma,K)$. 
Let us denote the associated relations on $\tilde{\Sigma}$ by $\relContTilde_\alpha$, $\alpha\geq 0$, and $\relContTilde \subset \relContTilde_0$. Let $\solContTilde$ be the sheaf of $\Omega$-regular, $\tilde{K}$-contact immersions. As noted earlier, $\solContTilde = \Sol(\relContTilde_\alpha)=\Sol(\relContTilde)$. Since $K = \tilde K|_\Sigma \cap T\Sigma$, the restriction morphism $C^\infty(\tilde{\Sigma},M) \to C^\infty(\Sigma, M)$ gives rise to a sheaf homomorphism
\begin{align*}
	ev : \solContTilde|_\Sigma &\to \solCont\\
	u &\mapsto u|_{\Sigma}
\end{align*}
which naturally induces a map $ev: \relContTilde|_{\Sigma} \to \relCont$. Therefore, $\relContTilde$ is an extension of $\relCont$. To keep the notation light, we have denoted the induced map by $ev$ as well.\medskip

\paragraph{\underline{\textit{Notation}}} For any subset $A$ in $\Sigma$, we use the notations $\Op A$ (resp., $\tilde{\Op} A$) to denote an arbitrary open set containing $A$ in $\Sigma$ (resp. $\tilde{\Sigma}$).

\begin{prop}\label{prop:relContExtensionHPrin}
	Let $O\subset \Sigma$ be a coordinate chart and $C\subset O$ be a compact subset. Suppose $U\subset M$ is an open subset such that $\mathcal{D}|_U$ is trivial. Then given any $\Omega$-regular $K$-contact immersion $u:\Op C \to U\subset M$, the 1-jet map 
	$$j^1 : ev^{-1}(u)\to ev^{-1}(F=j^1_u)$$
	in the commutative diagram,
	\[\begin{tikzcd}
		ev^{-1}(u) \arrow[hookrightarrow]{r} \arrow[dashed]{d}[swap]{j^1} & \solContTilde|_{C\times 0}\arrow{d} \arrow{r} & \solCont|_{C} \arrow{d} & u \arrow[maps to]{d}\\
		ev^{-1}(F) \arrow[hookrightarrow]{r} & \secContTilde|_{C \times 0} \arrow{r} &  \secCont|_{C} & F = j^1_u
	\end{tikzcd}\]
	induces a surjection between the set of path components.
\end{prop}
\begin{proof}
	Recall the following sheaves: $$\solCont = \Sol\relCont, \quad \secCont=\Gamma\relCont, \qquad \solContTilde=\Sol\relContTilde,\quad \secContTilde = \Gamma \relContTilde.$$  
	Fix some neighborhood $V$ of $C$, with $C\subset V\subset O$, over which $u$ is defined. The proof now proceeds through the following steps.
	\begin{description}[leftmargin=*]
		\item[Step $1$] \label{prop:relContExtensionHPrin:step1} Given an arbitrary extension $\tilde{F} \in \secContTilde|_{C\times 0}$ of $F$ along $ev$, we construct a regular solution $\bar u$ on $\tilde{Op}C$, so that $j^1_{\bar u}|_{\Op C} = \tilde{F}|_{\Op C}$.
		
		\item[Step $2$] \label{prop:relContExtensionHPrin:step2} We get a homotopy between $j^1_{\bar u}$ and $\tilde{F}$ in the affine bundle $J^1(W,U)$ which is constant on points of $C$.
		
		\item[Step $3$] \label{prop:relContExtensionHPrin:step3} We then push the homotopy obtained in \hyperref[prop:relContExtensionHPrin:step2]{Step $2$} inside $\relContTilde$, using \autoref{lemma:relContRetraction}. Thereby completing the proof.
	\end{description}
	
	\paragraph{\underline{Proof of \hyperref[prop:relContExtensionHPrin:step1]{Step $1$}}:} Suppose $\tilde{F}\in \secContTilde|_{C\times 0}$ is some arbitrary extension of $F$ along $ev$. Using \autoref{lemma:jetLiftingIsoCont}, we then get an arbitrary lift $\hat F\in \Gamma \relContTilde_\alpha|_C$ of $\tilde F$, for $\alpha$ sufficiently large (in fact, $\alpha\ge 4$ will suffice). The formal maps are represented in the following diagram.
	\[\begin{tikzcd}
		&\relContTilde_\alpha|_{\Op C} \arrow{d}{p^{\alpha+1}_1}\\
		&\relContTilde|_{\Op C} \arrow{d}{ev}\\
		\Op C \arrow{ruu}{\hat F} \arrow{ru}[swap]{\tilde{F}} \arrow{r}[swap]{F} &\relCont
	\end{tikzcd}\]
	We can now define a map $\hat u : \tilde{\Op}(C)\to U$ so that $j^{\alpha+1}_{\hat u}(p,0) = \hat F(p,0)$, by applying a Taylor series argument. In particular, we have $\hat u|_{C\times 0} = u$ and $\hat u$ is regular on points of $\Op(C)\times 0$. Since $C$ is a compact set and regularity is an open condition, we have that $\hat u$ is regular on some open set $W \subset \tilde\Sigma$ satisfying, $C\subset W\subset \bar W \subset \tilde \Op(C)$. Moreover, $\hat u$ is a regular infinitesimal solution along the set $W_0 = (V\times {0})\cap W\subset \tilde\Op(C)$ of order $$\alpha \ge 2.1 + 3.0 + \max\{1, 2.1 + 0\} = 4,$$
	for the equation $\tilde{\mathfrak{D}} = 0$, where $\tilde{\mathfrak{D}} = \opContTilde : v\mapsto v^*\lambda^s|_{\tilde{K}}$ is defined over $C^\infty(W, U)$. Now, by applying \autoref{thm:consistenInversion} we get an $\Omega$-regular immersion $\bar u: V\to U$ such that, $\tilde{\mathfrak{D}}(\bar u) = 0$ and furthermore, $$j^1_{\bar u} = j^1_{\hat u} \quad \text{on points of $W_0$.}$$
	In particular, $j^1_{\bar{u}}(p,0) = \tilde F(p,0)$ for $(p,0)\in W_0$ and so $u$ on $\Op C$ is extended to $\bar{u}$ on $W$. \smallskip
	
	\paragraph{\underline{Proof of \hyperref[prop:relContExtensionHPrin:step2]{Step $2$}}:} Let us denote $\tilde u = \bs \tilde{F}$ and define, $v_t(x,s) = \bar u (x, ts)$ for $(x,s)\in W$. Note that, $$v_0(x,s) = \bar u(x,0) = \hat u(x,0) = \tilde{u}(x, 0)$$
	and so $v_t$ is a homotopy between the maps $\bar u$ and $\pi^*(\tilde{u}|_{\Op C})|_W$, where $\pi : \Sigma\times \mathbb{R} \to \Sigma$ is the projection. Now, with the help of some auxiliary choice of parallel transport on the vector bundle $J^1(W,U)$, we can get isomorphisms
	$$\varphi(x,s) : J^1_{((x,0), \bar u(x,0))}(W, U) \to J^1_{((x,s), \bar u(x,s))}(W, U), \quad \text{for $(x,s)\in W$, for $s$ sufficiently small,}$$
	so that $\varphi(x,0) = \Id$. We then define the homotopy, $$G_t|_{(x,s)} = (1-t) \cdot \varphi(x,ts) \circ  \tilde{F}|_{(x,0)} + t \cdot j^1_{\bar u}(x,ts) \; \in J^1_{((x,ts), \bar u(x,ts))}(W,U).$$
	Clearly $G_t$ covers $v_t$; we have $$G_0|_{(x,s)} = \varphi(x,0) \circ \tilde{F}|_{(x,0)} = \tilde{F}|_{(x,0)} = \tilde{F}|_{(x,s)} \qquad\text{and}\qquad G_1|_{(x,s)} = j^1_{\bar u}(x,s).$$
	Thus, we have obtained a homotopy $G_t$ between $\pi^*(\tilde{F}|_{\Op C})|_{\tilde{Op} C}$ and $j^1_{\bar u}$. Similar argument produces a homotopy between $\tilde{F}$ and $\pi^*(\tilde{F}|_{\Op C})|_W$ as well. Concatenating the two homotopies, we have a homotopy $H_t$ between $\tilde{F}$ and $j^1_{\bar u}$, in the affine bundle $J^1(W, U)\to W\times U$. However, $H_t$ need not lie in $\relContTilde$.\smallskip
	
	\paragraph{\underline{Proof of \hyperref[prop:relContExtensionHPrin:step3]{Step $3$}}:} By \autoref{lemma:relContRetraction}, we get a tubular neighborhood $\mathcal{N}\subset J^1(W,U)$ of $\relContTilde$ which \emph{fiber-wise} deformation retracts onto $\relContTilde$. Suppose $\rho : \mathcal{N}\to \relContTilde$ is such a retraction. Now, note that on points of $C$ $$H_t|_{(x,0)} = (1-t) \cdot \tilde F|_{(x,0)} + t\cdot j^1_{\bar u}(x,0) = \tilde{F}|_{(x,0)}.$$
	Since $C$ is compact, we may get a neighborhood $W'$, satisfying $C\subset W^\prime \subset W$, such that the homotopy $H_t|_{W^\prime}$ takes its values in the open neighborhood $\mathcal{N}$ of $\im\tilde F$. Then composing with the retraction $\rho$, we can push this homotopy inside the relation $\relContTilde$, obtaining a homotopy $\tilde{F}_t\in\tilde\Psi|_C$ joining $\tilde F$ to $j^1_{\bar u}$. Observe that the homotopy remains constant on points of $C$. In particular, $ev(\tilde F_t) = F$ on points of $C$. This concludes the proof.
\end{proof}

We have the following $h$-principle for $\relCont$.

\begin{theorem}\label{thm:hPrinGenExtnHoriz}
	Suppose $(\Sigma, K)$ is a pair consisting of an arbitrary manifold $\Sigma$ with a distribution $K$. Let $(\tilde{\Sigma},\tilde{K})=(\Sigma\times\mathbb{R}, K\times\mathbb{R})$. Then $\relContTilde$ is an extension of $\relCont$, and $\solContTilde|_\Sigma$ is a flexible sheaf. If $\relContTilde$ is a microextension of $\relCont$ (\autoref{defn:microextension}), then the relation $\relCont$ satisfies the $C^0$-dense parametric $h$-principle.
\end{theorem}
\begin{proof}
	Let us take any $\phi \in \diff(\tilde\Sigma,\pi)$ so that  $\pi \circ \phi = \pi$, where $\pi : \tilde \Sigma \to \Sigma$ is the projection (\autoref{exmp:sharplyMovingDiffeo} (\ref{exmp:sharplyMovingDiffeo:1})). Then, we have
		\[d\phi^{-1}(\tilde K) = d\phi^{-1}\left( d\pi^{-1} (K) \right) = d(\pi\circ \phi)^{-1}(K) = d\pi^{-1}(K) = \tilde K.\]
	This shows that $\solContTilde$ is invariant under the natural action of $\diff(\tilde\Sigma,\pi)$. Furthermore, diffeotopies in $\diff(\tilde\Sigma,\pi)$ sharply move $\Sigma_0\cong \Sigma\times 0$ in $\Sigma\times\mathbb R$. Then, the flexibility of $\solContTilde|_{\Sigma_0}$ follows from \autoref{thm:mainFlexibilityTheorem}. The parametric $h$-principle for $\relCont$ can now be derived from the next lemma by a standard argument as in \cite{eliashbergBook}.
\end{proof}

\begin{lemma}\label{lemma:hPrinParametricRelative}
	Suppose, for any contractible open set $O\subset \Sigma$, the map $ev : \Gamma \relContTilde|_O \to \Gamma \relCont|_O$ is surjective. For a compact polyhedron $P$ along with a subpolyhedron $Q \subset P$, we are given the formal data $\{F_z\}_{z\in P} \in \Gamma \relCont$ such that $F_z$ is holonomic for $z \in \Op Q \subset P$. Then, there exists a homotopy $F_{z,t} \in \Gamma \relCont$ satisfying the following:
	\begin{itemize}
		\item $F_{z,0} = F_z$ for all $z \in P$.
		\item $F_{z,1}$ is holonomic for all $z \in P$.
		\item $F_{z, t} = F_z$ for all time $t$, if $z \in \Op Q$.
	\end{itemize}
	Furthermore, the homotopy can be chosen to be arbitrary $C^0$-small in the base.
\end{lemma}
\begin{proof}
	The proof is essentially done via a cell-wise induction. Let us denote the base maps $u_z = \bs F_z$. \smallskip
	
	\paragraph{\underline{Setup}:} First, we fix a cover $\mathcal{U}$ of $M$ by open balls, so that $\mathcal{D}|_U$ is cotrivial for each $U\in \mathcal{U}$. Next, fix a `good cover' $\mathcal{O}$ of $\Sigma$ subordinate to the open cover $\{u_z^{-1}(U) \;|\; U\in\mathcal{U}, \, z \in P\}$. By a \emph{good cover}, we mean that $\mathcal{O}$ consists of contractible open charts of $\Sigma$, which is closed under finite (non-empty) intersections. Then, fix a triangulation $\{\Delta^\alpha\}$ of $\Sigma$ subordinate to $\mathcal{O}$. For each top-dimensional simplex $\Delta^\alpha$ choose $O_\alpha \in \mathcal{O}$ such that $\Delta^{\alpha}\subset O_{\alpha}$. For any other simplex $\Delta$ we denote
	$$O_\Delta = \bigcap_{\Delta\subset\Delta^\beta} O_\beta, \;\text{where $\Delta^\beta$ is top-dimensional.}$$
	Since any $\Delta$ is contained in at most finitely many simplices, $O_\Delta\in\mathcal{O}$ as it is a good cover. For each $z \in P$, let us also fix $U_{z,\Delta} \in \mathcal{U}$ so that $O_\Delta \subset u_z^{-1}(U_{z,\Delta})$. Throughout the proof, for any fixed simplex $\Delta$ we shall assume $\Op\Delta \subset O_\Delta$. Furthermore, we shall assert that the base maps of any homotopy of $F_{z}$ near $\Delta$ has their value in $U_{z,\Delta}$. Thus, the $C^0$-smallness of the homotopy can be controlled by a priori choosing the open cover $\mathcal{U}$ sufficiently small. \smallskip
	
	\paragraph{\underline{Induction Base Step}:} Fix a $0$-simplex $v\in \Sigma$. The map $j^1: \solCont \to \secCont$ is a \emph{local} weak homotopy equivalence by \autoref{corr:relContLocalWHE}. In particular, $j^1 : \solCont|_v \to \secCont|_v$ is a weak homotopy equivalence and consequently, we have a homotopy $F_{z,t}^v \in \secCont$ defined over $\Op(v)$ so that $$\text{$F_{z,0}^v =F_z$ and $F_{z,1}^v$ is holonomic on $\Op(v)$, for $z \in P$}.$$
	Furthermore, we can arrange so that the map $j^1$ is a homotopy equivalence, so that we get $F_{z,t} = F_z$ for all $z \in \Op Q$ as well.
	
	Now, by a standard argument using cutoff functions, we patch all these homotopies and get a homotopy $F_{z,t}^0 \in \secCont$ satisfying $$\text{$F_{z,1}^0$ is holonomic on $\Op\Sigma^{(0)}$ and $F_{z,t}^0 = F_z$ on $\Sigma\setminus \Op\Sigma^{(0)}$,}$$
	where $\Sigma^{(0)}$ is the $0$-skeleton of $\Sigma$. Furthermore, $F_{z,t} = F_z$ for $z \in \Op Q$ by construction. \smallskip
	
	\paragraph{\underline{Induction Hypothesis}:} Suppose for $i \ge 0$, we have obtained the homotopy $F_{z,t}^i \in \secCont$ so that $$\text{$F_{z,0}^i = F_{z,1}^{i-1}$,\; $F_{z,1}^i$ is holonomic on $\Op\Sigma^{(i)}$ \; and \; $F_{z,t}^i = F_{z,1}^{i-1}$ on $\Sigma \setminus \Op\Sigma^{(i)}$ for $t\in[0,1], \; z \in P$,}$$
	where $\Sigma^{(i)}$ is the $i$-skeleton of $\Sigma$. Furthermore, $F_{z,t}^i = F_z$ for $z \in \Op Q$ and the homotopies are arbitrary $C^0$-small in the base maps. For notational convenience, we set $F_{z,1}^{-1} = F_z$. \smallskip
	
	\paragraph{\underline{Induction Step}:} Fix an $i+1$-simplex $\Delta$. By the hypothesis of the theorem, we first obtain some arbitrary lifts $\tilde F_z^\Delta \in \secContTilde|_\Delta$ of $F_{z,1}^i|_{\Op\Delta} \in\secCont|_\Delta$, along the map $ev$. The lifts can be chosen to be continuous with respect to the parameter space $z \in P$. Since $F_{z,1}^i|_{\Op\partial\Delta}$ is holonomic by the induction hypothesis, applying \autoref{prop:relContExtensionHPrin} for the compact set $C=\partial\Delta$, we obtain a homotopy $$\tilde G_{z,t}^{\partial\Delta}\in\secContTilde|_{\partial\Delta}$$
	joining $\tilde F_z^\Delta|_{\Op\partial\Delta}$ to a \emph{holonomic} section $\tilde G_{z,1}^{\partial\Delta}\in \secContTilde|_{\partial\Delta}$. Furthermore, the homotopy satisfies $ev(\tilde G_{z,t}^{\partial \Delta}) = F_{z,1}^i|_{\Op \partial\Delta}$ for $t \in [0,1]$ and for $z \in P$. Using the flexibility of the sheaf $\secContTilde|_\Sigma$ we extend $\tilde G_{z,t}^{\partial\Delta}$ to a homotopy $\tilde G_{z,t}^\Delta \in \secContTilde|_\Delta$ defined on some $\tilde\Op\Delta$, so that
	$$\text{ $\tilde G_{z,1}^\Delta|_{\tilde\Op\partial\Delta} = \tilde G_{z,1}^{\partial\Delta}$ is holonomic, for $z \in P$.}$$
	Furthermore, we can arrange so that $\tilde{G}^\Delta_{z,t} = \tilde{F}_{z}^\Delta$ for all time $t$, whenever $z \in \Op Q$. Denoting $\tilde G_{z,1}^\Delta|_{\tilde \Op \partial\Delta} = j^1_{\tilde u_z^{\partial\Delta}}$ for smooth maps $\tilde u_z^{\partial\Delta}$ defined on $\tilde\Op\partial\Delta$, we consider the map of \emph{fibrations} as follows.
	\[\begin{tikzcd}
			\eta^{-1}\big(\tilde u_z^{\partial\Delta}\big) \arrow[hookrightarrow]{r} \arrow{d}[swap]{J} & \solContTilde|_{\Delta} \arrow{r}{\eta} \arrow{d}{J} & \solContTilde|_{\partial\Delta} \arrow{d}{J} & \tilde{u}_z^{\partial\Delta}\arrow[maps to]{d}\\
			\chi^{-1}\big(\tilde{G}_{z,1}^\Delta|_{\tilde\Op\partial\Delta}\big) \arrow[hookrightarrow]{r} & \secContTilde|_{\Delta} \arrow{r}[swap]{\chi} & \secContTilde|_{\partial\Delta} &  j^1_{\tilde u_z^{\partial\Delta}} = \tilde G^\Delta_{z,1}|_{\tilde\Op\partial\Delta}
		\end{tikzcd}\]
	Here $\eta$ is indeed a fibration, as $\solContTilde|_\Sigma$ is flexible by \autoref{thm:mainFlexibilityTheorem}. Now, the rightmost and the middle $J=j^1$ are \emph{local} weak homotopy equivalences by \autoref{thm:microflexibleLocalWHESheafTheorem} and \autoref{lemma:jetLiftingIsoCont}. Hence, they are in fact weak homotopy equivalences by an application of the sheaf homomorphism theorem (\autoref{thm:sheafHomoTheorem}). By the $5$-lemma argument, we then have $$J:\eta^{-1}(\tilde u_z^{\partial\Delta}) \to \chi^{-1}\big(\tilde{G}_{z,1}^\Delta|_{\tilde\Op\partial\Delta}\big)$$
	is a weak homotopy equivalence. Now, $\tilde{G}_{z,1}^\Delta \in \chi^{-1}\big(\tilde{G}_{z,1}^\Delta|_{\tilde\Op\partial\Delta}\big)$. Hence, we have a path $\tilde H_{z,t} \in \chi^{-1}\big(\tilde{G}_{z,1}^\Delta|_{\tilde\Op\partial\Delta}\big)$ joining $\tilde{G}_{z,1}^{\Delta}$ to some \emph{holonomic} section $\tilde H_{z,1}$. In particular, this homotopy is fixed on $\tilde\Op\partial\Delta$. Furthermore, we can make sure that this weak homotopy equivalence is in fact a homotopy equivalence and hence, we can assume $H_{z,t}$ is constant for $z \in \Op Q$. We get the concatenated homotopy
	$$\tilde{F}_{z,t} : \tilde F_z^\Delta \sim_{\tilde G_{z,t}^\Delta} \tilde G_{z,1}^\Delta \sim_{\tilde H_{z,t}} \tilde H_{z,1}, \quad z \in P,$$
	and set $F_{z,t}^\Delta = ev(\tilde F_{z,t})$. Clearly, $F_{z,0}^\Delta = F_{z,1}^{i}$ on $\Op\partial\Delta$ and $F_{z,1}^\Delta$ is holonomic on $\Op\Delta$. Also, $F_{z,t} = F_z$ for $z \in \Op Q$.
	
	Using a standard cutoff function argument, we patch these homotopies together and get the homotopy $F_{z,t}^{i+1}\in\secCont$ satisfying $$\text{$F_{z,0}^{i+1} = F_{z,1}^i$, \; $F_{z,1}^{i+1}$ is holonomic on $\Op\Sigma^{(i)}$, \; and $F_{z,t}^{i+1}=F_{z,1}^{i}$ on $\Sigma \setminus \Op\Sigma^{(i+1)}$,}$$
	where $\Sigma^{(i+1)}$ is the $i+1$-skeleton of $\Sigma$. Furthermore, $F_{z,t} = F_z$ for $z \in \Op Q$. \smallskip
	
	The induction terminates once we have obtained the homotopy $F_{z,t}^k$, where $k = \dim\Sigma$. We end up with a sequence of homotopies in $\secCont$. Concatenating all of them we have the homotopy
	$$F_{z,t} : F_z = F_{z,1}^{-1} \sim_{F_{z,t}^0} F_{z,1}^0 \sim_{F_{z,t}^1} F_{z,1}^1\sim\cdots\sim_{F_{z,t}^{k-1}} F_{z,1}^{k-1}\sim_{F_{z,t}^k} F_{z,1}^k.$$
	Clearly $F_{z,t}\in\secCont$ is the desired homotopy joining $F_z$ to a \emph{holonomic} section $F_{z,1}=F_{z,1}^k\in\secCont$, where $F_{z,t} = F_z$ for $z \in \Op Q$. Since at each step the homotopy can be chosen to be arbitrary $C^0$-small and since there are only finitely many steps, we see that $F_{z,t}$ can be made arbitrary $C^0$-small in the base map as well. This concludes the proof.
\end{proof}

\begin{remark}
	In \cite{duPlessisHPrinciple}, the author has obtained similar $h$-principle for \emph{open} relations which admit $\diff$-invariant ``open extensions''. We also refer to \cite[pg. 127-128]{eliashbergBook} where parametric $h$-principle is obtained under stronger hypothesis.
\end{remark}

\subsubsection{$h$-principle for $\relICont$}
In order to obtain the $h$-principle for the relation $\relICont$, we consider the class $\mathcal{S}^\prime$ of maps $u : \Sigma \to M$, which are $(d\lambda^s)$-regular immersions (\autoref{defn:contOpenRegular}), and furthermore satisfy the (open) condition $$\rk (d\lambda^s \circ du) \ge \cork K.$$
Clearly, $\mathcal{S}^\prime$ is the solution space of an open relation $S^\prime \subset J^1(\Sigma, M)$, where $S^\prime$ is a subrelation of $S$. Note that any $K$-contact immersion satisfying the above inequality is $K$-isocontact. Conversely, every $K$-isocontact immersion satisfies the above inequality (compare \hyperref[defn:relCont:4]{(4)} of \autoref{defn:relCont}). Thus, the $S^\prime$-regular solutions of the operator $\opCont$ are precisely the $\Omega$-regular $K$-\emph{iso}contact immersions $\Sigma \to M$. Proceeding as in the case of $\relCont$, we can prove \autoref{thm:hPrinGenExtnHoriz} for $\relICont$ as well.

\subsubsection{$h$-principle for $\relHor$}
For the special case of $\relHor$ (for which $K=T\Sigma$), \autoref{thm:hPrinGenExtnHoriz} can be compared with the $h$-principle for `overregular' maps (Approximation Theorem in \cite[pg. 258]{gromovCCMetric}). In general, $ev : \relHorTilde|_\Sigma\to \relHor$ fails to be surjective, and \emph{overregular maps} are precisely the solutions to $ev\big(\relHorTilde|_\Sigma\big)$.
	
If $\mathcal{D}$ is a contact distribution then every horizontal immersion $\Sigma\to (M,\mathcal D)$ is regular. Moreover, one does not require overregularity condition to obtain the $h$-principle for Legendrian immersions \cite{duchampLegendre}. Indeed, by embedding $\Sigma$ in the contact manifold $\tilde{\Sigma}=J^1(\Sigma,\mathbb R)$ one observes that $\relIContTilde$ is a microextension of $\relHor$ and hence $\relHor$ satisfies the $h$-principle.\medskip

In general, we can prove the following.
\begin{theorem}\label{thm:hPrinGenExtnContact}
	Suppose $\Sigma$ is an arbitrary manifold. Let $(\tilde{\Sigma},\tilde{K})=(J^1(\Sigma,\mathbb{R}), \xi_{std})$ denote the first jet space, with the canonical contact structure. Then $\relIContTilde$ is an extension of $\relHor$. If $\relIContTilde$ is a microextension of $\relHor$ then the relation $\relHor$ satisfies the $C^0$-dense $h$-principle.
\end{theorem}
\begin{proof}
	Since $\Sigma$ canonically embeds in the 1-jet space $\tilde{\Sigma}=J^1(\Sigma,\mathbb{R})$ as a Legendrian submanifold, it is easy to note that $\relIContTilde$ is an extension of $\relHor$. The sheaf $\solIContTilde$ is invariant under the natural action of contactomorphisms of $J^1(\Sigma,\mathbb R)$; furthermore, by \autoref{exmp:sharplyMovingDiffeo} (\ref{exmp:sharplyMovingDiffeo:2}), the contact diffeotopies of $J^1(\Sigma,\mathbb R)$ sharply move $\Sigma$ in $\tilde \Sigma$. Then, the flexibility of $\solIContTilde|_{\Sigma_0}$ follows from \autoref{thm:mainFlexibilityTheorem}. Arguing as in the proof of \autoref{lemma:hPrinParametricRelative} we obtain the $h$-principle for $\relHor$.
\end{proof}

\subsection{Proof of \autoref{lemma:jetLiftingIsoCont}}\label{sec:jetLiftingLemmaProof}
To simplify the notation, we assume that $K=T\Sigma$, i.e., we prove the statement for the relation $\relHor$. The argument for a general $K$ is similar, albeit cumbersome. As the lemma is local in nature, without loss of generality we assume $\mathcal{D}$ is cotrivializable and hence let us write $\mathcal{D}=\bigcap_{s=1}^p\ker\lambda^s$ for $1$-forms $\lambda^1,\ldots,\lambda^p$ on $M$. Denote the tuples $$\lambda = (\lambda^s)\in\Omega^1(M,\mathbb{R}^p) \text{ and } d\lambda = (d\lambda^s)\in\Omega^2(M,\mathbb{R}^p).$$
We need to consider the three operators:
$$u \mapsto u^*\lambda, \qquad u\mapsto u^*d\lambda, \qquad\text{the exterior derivative operator, $d : \Omega^1(\Sigma,\mathbb{R}^p) \to \Omega^2(\Sigma,\mathbb{R}^p)$.}$$
Their respective symbols are as follows.
\begin{itemize}
	\item We have the bundle map $\Delta_\lambda : J^1(\Sigma,M) \to \Omega^1(\Sigma, \mathbb{R}^p)$ so that, $\Delta_\lambda\big(j^1_u\big) = u^*\lambda = \big(u^*\lambda^s\big)$. Explicitly, $$\Delta_\lambda(x,y, F:T_x\Sigma\to T_y M) = \big(x, F\circ\lambda|_y\big).$$
	
	\item We have the bundle map $\Delta_{d \lambda} : J^1(\Sigma, M) \to \Omega^1(\Sigma, \mathbb{R}^p)$ so that, $\Delta_{d\lambda}\big(j^1_u\big) = u^*d\lambda = \big(u^*d\lambda^s\big)$. Explicitly, $$\Delta_{d\lambda}(x,y, F:T_x\Sigma\to T_y M) = \big(x, F^*d\lambda|_y\big).$$
	
	\item We have the bundle map $\Delta_d : \Omega^1(\Sigma,\mathbb{R}^p)^{(1)} \to \Omega^2(\Sigma,M)$ so that, $\Delta_d(j^1_\alpha) = d\alpha$. Explicitly, $$\Delta_d\big(x,\alpha, F:T_x\Sigma\to \hom(T_x\Sigma,\mathbb{R}^p)\big) = \big(x, (X\wedge Y) \mapsto F(X)(Y) - F(Y)(X)\big).$$
\end{itemize}

\paragraph{\bfseries Jet Prolongation of Symbols:} Recall that given some arbitrary $r^\text{th}$-order operator $\mathfrak{D}:\Gamma X\to \Gamma G$ represented by the symbol $\Delta : X^{(r)} \to G$ as, $\Delta(j^r_u) = \mathfrak{D}(u)$, we have the $\alpha$-jet prolongation, $\Delta^{(\alpha)} : X^{(r+\alpha)} \to G^{(\alpha)}$ defined as $$\Delta^{(\alpha)}(j^{r+\alpha)}_u(x)) = j^\alpha_{\mathfrak{D}(u)}(x).$$
Then, for any $\alpha\ge \beta$ we have $p^{\alpha}_\beta \circ \Delta^{(\alpha)} = \Delta^\beta \circ p^{r + \alpha}_{r+\beta}$. Let us observe the following interplay between the symbols of the operators introduced above.
\begin{itemize}
	\item We have the commutative diagram
	\[\begin{tikzcd}
		J^{\alpha+1}(\Sigma, M) \arrow{r}{\Delta_\lambda^{(\alpha)}} \arrow{d}[swap]{p^{\alpha+1}_\alpha} & \Omega^1(\Sigma, \mathbb{R}^p)^{(\alpha)} \arrow{d}{\Delta_d^{(\alpha-1)}} \\
		J^{\alpha}(\Sigma, M) \arrow{r}[swap]{\Delta_{d\lambda}^{(\alpha-1)}} & \Omega^2(\Sigma, \mathbb{R}^p)^{(\alpha-1)}
	\end{tikzcd}\]
	Indeed, we observe $$\Delta_d^{(\alpha-1)}\circ \Delta_\lambda^{(\alpha)}\big(j^{\alpha+1}_u(x)\big) = \Delta_d^{(\alpha-1)}\big(j^\alpha_{u^*\lambda}(x)\big) = j^{\alpha-1}_{d\big(u^*\lambda\big)}(x) = j^{\alpha-1}_{u^*d\lambda}(x) = \Delta_{d\lambda}^{(\alpha-1)}\big(j^{\alpha}_u(x)\big),$$
	and hence, we get $$\Delta_d^{(\alpha-1)}\circ \Delta_\lambda^{(\alpha)} = \Delta_{d\lambda}^{(\alpha-1)}\circ p^{\alpha+1}_\alpha.$$
	
	\item We have the two commutative diagrams
	\[\begin{tikzcd}
		J^{\alpha+1}(\Sigma,M) \arrow{r}{\Delta_\lambda^{(\alpha)}} \arrow{d}[swap]{p^{\alpha+1}_\alpha} & \Omega^1(\Sigma,\mathbb{R}^p)^{(\alpha)} \arrow{d}{p^\alpha_{\alpha-1}} \\
		J^\alpha(\Sigma,M) \arrow{r}[swap]{\Delta_\lambda^{(\alpha-1)}} & \Omega^1(\Sigma,\mathbb{R}^p)^{(\alpha-1)}
	\end{tikzcd}\quad \text{and} \quad
	\begin{tikzcd}
		J^{\alpha+1}(\Sigma,M) \arrow{r}{\Delta_{d\lambda}^{(\alpha)}} \arrow{d}[swap]{p^{\alpha+1}_\alpha} & \Omega^2(\Sigma,\mathbb{R}^p)^{(\alpha)} \arrow{d}{p^\alpha_{\alpha-1}} \\
		J^\alpha(\Sigma,M) \arrow{r}[swap]{\Delta_{d\lambda}^{(\alpha-1)}} & \Omega^2(\Sigma,\mathbb{R}^p)^{(\alpha-1)}
	\end{tikzcd}\]
\end{itemize}

Next, let us fix $\mathcal{R}_{d\lambda} \subset J^1(\Sigma, M)$ representing the $(d\lambda^s)$-regular immersions $\Sigma \to M$, i.e,
$$\mathcal{R}_{d\lambda} = \Big\{(x,y,F:T_x\Sigma \to T_yM) \;\Big|\; \text{ $F$ is injective and $(d\lambda^s)$-regular} \Big\}.$$
Recall that $\mathcal{R}_\alpha := \relHor_\alpha \subset J^{\alpha+1}(\Sigma,M)$ is given as,
$$\mathcal{R}_\alpha = \Big\{j^{\alpha+1}_u(x)\in J^{\alpha+1}(\Sigma,M)|_x \;\Big|\; \text{$j^\alpha_{u^*\lambda}(x) = 0$ and $u$ is $(d\lambda^s)$-regular}\Big\}.$$
Hence, we can identify $\mathcal{R}_\alpha$ as $$\mathcal{R}_\alpha = \ker\big(\Delta_\lambda^{(\alpha)}\big) \cap \big(p^{\alpha+1}_1\big)^{-1}(\mathcal{R}_{d\lambda}) \subset J^{\alpha+1}(\Sigma,M),$$
where $p^{\alpha+1}_1 :J^{\alpha+1}(\Sigma,M)\to J^1(\Sigma,M)$ is the natural projection map. We denote a sub-relation, $$\bar{\mathcal{R}}_\alpha = \mathcal{R}_\alpha \cap \ker\big(\Delta_{d\lambda}^{(\alpha)}\big) \subset \mathcal{R}_\alpha.$$
In particular, observe that $\bar{\mathcal{R}}_0$ is then precisely $\relHor$, i.e, the relation of $\Omega$-regular, horizontal immersions $\Sigma\to M$. The proof of \autoref{lemma:jetLiftingIsoCont} follows from the next two results.
\begin{sublemma}\label{lemma:jetLiftingIsoCont:Lambda}
	For any $\alpha\ge 0$, we have, $\bar{\mathcal{R}}_\alpha = p^{\alpha+2}_{\alpha+1}\big(\mathcal{R}_{\alpha+1}\big)$ and for each $(x,y)\in\Sigma\times M$, the fiber of $p^{\alpha+2}_{\alpha + 1} : \mathcal{R}_{\alpha+1}|_{(x,y)}\to \bar{\mathcal{R}}_\alpha|_{(x,y)}$ is affine. Furthermore, any section of $\bar{\mathcal{R}}_{\alpha}|_O$, over some contractible charts $O\subset \Sigma$, can be lifted to a section of $\mathcal{R}_{\alpha+1}|_O$ along $p^{\alpha+2}_{\alpha+1}$.
\end{sublemma}
\begin{sublemma}\label{lemma:jetLiftingIsoCont:OmegaRegularity}
	For any $\alpha\ge 0$, the map $p^{\alpha+2}_{\alpha +1}:\bar{\mathcal{R}}_{\alpha+1}|_{(x,y)} \to \bar{\mathcal{R}}_\alpha|_{(x,y)}$ is surjective, with affine fibers, for each $(x,y)\in\Sigma\times M$. Furthermore, any section of $\bar{\mathcal{R}}_\alpha|_O$ over some contractible chart $O\subset\Sigma$ can be lifted to a section of $\bar{\mathcal{R}}_{\alpha+1}|_O$ along $p^{\alpha+2}_{\alpha+1}$.
\end{sublemma}

\begin{proof}[Proof of \autoref{lemma:jetLiftingIsoCont}]
	We have the following ladder-like schematic representation of the proof.
	\[
	\tikzset{%
		symbol/.style={
			draw=none,
			every to/.append style={
				edge node={node [sloped, allow upside down, auto=false]{$#1$}}
			},
		},
	}
	\begin{tikzcd}
		J^{\alpha+1}(\Sigma,M) \arrow{r}{p^{\alpha+1}_\alpha} & J^{\alpha}(\Sigma,M) \arrow{r} &\cdots \arrow{r} & J^2(\Sigma,M) \arrow{r}{p^2_1} & J^1(\Sigma,M)\\
		\mathcal{R}_{\alpha} \arrow[rdd, twoheadrightarrow, sloped, "p^{\alpha+1}_{\alpha}"] \arrow{r} \arrow[symbol=\subset]{u} & \mathcal{R}_{\alpha-1} \arrow{r} \arrow[symbol=\subset]{u} & \cdots \arrow{r} & \mathcal{R}_1 \arrow{r} \arrow[twoheadrightarrow]{rdd} \arrow[symbol=\subset]{u} & \mathcal{R}_0 \arrow[symbol=\subset]{u} \\
		\\
		& \bar{\mathcal{R}}_{\alpha-1} \arrow[hookrightarrow]{uu} \arrow[dashed, bend left=40, sloped]{uul}{\text{lift using}}[swap]{ \substack{\text{full rank of $\lambda$} \\ \text{(\autoref{lemma:jetLiftingIsoCont:Lambda})}} } &&& \relHor = \bar{\mathcal{R}}_0 \arrow[uu, hookrightarrow, sloped] \arrow[dashed]{lll}[swap]{\text{lift inductively to $\bar{\mathcal{R}}_{\alpha-1}$}}{\text{using $\Omega$-regularity (\autoref{lemma:jetLiftingIsoCont:OmegaRegularity})} }
	\end{tikzcd}
	\]
	For any $\alpha\ge 1$, we have $p^{\alpha+1}_1 = p^{\alpha}_1 \circ p^{\alpha+1}_\alpha = p^2_1 \circ \cdots \circ p^{\alpha+1}_\alpha$. From \autoref{lemma:jetLiftingIsoCont:Lambda} we have that $p^{\alpha+1}_{\alpha}$ maps $\mathcal{R}_{\alpha}$ surjectively onto $\bar{\mathcal{R}}_{\alpha-1}$. Also, using \autoref{lemma:jetLiftingIsoCont:OmegaRegularity} repeatedly, we have that $p^{\alpha}_1:\bar{\mathcal{R}}_{\alpha-1}\to \relHor$ is a surjection as well. Combining the two, we have the claim.
	
	Since at each step, we have contractible fiber, we see that the fiber of $p^{\alpha+1}_1$ is again contractible. In fact, we are easily able to get lifts of sections over contractible charts as well. This concludes the proof.
\end{proof}

We now prove the above sublemmas.
\begin{proof}[Proof of \autoref{lemma:jetLiftingIsoCont:Lambda}]
	We have the following commutative diagram
	\[\begin{tikzcd}
		\mathcal{R}_{\alpha+1}\arrow[hookrightarrow]{r} & J^{\alpha+2}(\Sigma,M) \arrow{rr}{\Delta_\lambda^{(\alpha+1)}} \arrow{d}[swap]{p^{\alpha+2}_{\alpha+1}} && \Omega^1(\Sigma, \mathbb{R}^p)^{(\alpha+1)} \arrow{d}[swap]{p^{\alpha+1}_\alpha} {\Delta_d^{(\alpha)}} \\
		\bar{\mathcal{R}}_{\alpha} \arrow[hookrightarrow]{r} & J^{\alpha+1}(\Sigma,M) \arrow{rr}[swap]{\Delta_\lambda^{(\alpha)}, \; \Delta_{d\lambda}^{(\alpha)}} && \Omega^1(\Sigma,\mathbb{R}^p)^{(\alpha)} \oplus \Omega^2(\Sigma, \mathbb{R}^p)^{(\alpha)}
	\end{tikzcd} \label{cd:jetLifting:surjectivityOfLambda}\tag{$*$} \]
	Since we have $\mathcal{R}_{\alpha+1}\subset \ker \Delta_\lambda^{(\alpha+1)}$, we get $$p^{\alpha+2}_{\alpha+1}(\mathcal{R}_{\alpha+1}) \subset \ker \Delta_\lambda^{(\alpha)} \cap \ker\Delta_{d\lambda}^{(\alpha)}.$$
	Also, we have $$\mathcal{R}_{\alpha+1}\subset \big(p^{\alpha+2}_1\big)^{-1}(\mathcal{R}_{d\lambda}) \Rightarrow p^{\alpha+2}_{\alpha+1}\big(\mathcal{R}_{\alpha+1}\big) \subset \big(p^{\alpha+1}_\alpha\big)^{-1}(\mathcal{R}_{d\lambda}).$$
	Hence, we see that $\big(p^{\alpha+2}_{\alpha+1}\big)(\mathcal{R}_{\alpha+1})\subset\bar{\mathcal{R}}_\alpha$.
	
	Conversely, let us assume that we are given a jet $$\big(x,y, P_i :\Sym^i T_x\Sigma \to T_y M, \; i=1, \ldots, \alpha +1 \big) \in \bar{\mathcal{R}}_\alpha|_{(x,y)}.$$
	We wish to find $Q:\Sym^{\alpha+2}T_x\Sigma \to T_y M$ so that $$(x,y, P_i, Q)\in\mathcal{R}_{\alpha+1}|_{(x,y)}.$$
	Recall that $\Delta_\lambda(x,y,F:T_x\Sigma\to T_y M) = \big(x, \lambda|_y \circ F : T_x\Sigma\to \mathbb{R}^p\big)$. Then, we may write $$\Delta_\lambda^{(\alpha+1)}\big(x,y,P_i,Q\big) = \big(x, \lambda\circ F, R_i : \Sym^i T_x\Sigma \to \hom(T_x\Sigma, \mathbb{R}^p),\; i=1,\ldots,\alpha+1\big),$$
	so that $R_{\alpha+1}: \Sym^{\alpha+1} T_x \Sigma \to \hom(T_x\Sigma, \mathbb{R}^p)$ is the \emph{only} symmetric tensor which involves $Q$. In fact, we observe that $R_{\alpha+1}$ is given explicitly as $$R_{\alpha+1}\big(X_1,\ldots,X_{\alpha+1}\big) (Y) = \lambda\circ Q\big(X_1,\ldots,X_{\alpha+1}, Y\big) + \text{terms involving $P_i$.}$$
	Now, from the commutative diagram \hyperref[cd:jetLifting:surjectivityOfLambda]{$(*)$} we have
	\begin{align*}
		\big(x,\lambda\circ F, R_i\; i=1, \ldots, \alpha\big) &= p^{\alpha+1}_\alpha\circ \Delta_\lambda^{(\alpha+1)}(x,y,P_i,Q) \\
		&= \Delta_\lambda^{(\alpha)}\circ p^{\alpha+2}_{\alpha+1}(x,y,P_i,Q) \\
		&= \Delta_\lambda^{(\alpha)}(x,y,P_i) \\
		&= 0.
	\end{align*}
	That is, we have, $R_i = 0$ for $i=1,\ldots,\alpha$. We need to find $Q$ so that $R_{\alpha+1}=0$ as well. We claim that the tensor $$R_{\alpha+1}^\prime : \big(X_1,\ldots,X_{\alpha+1},Y\big)\mapsto R_{\alpha+1}(X_1,\ldots,X_{\alpha+1})(Y),$$
	is symmetric.
	
	Let us write $\Delta_d^{(\alpha)}(x,y,\lambda\circ F,R_i) = \big(x,\omega, S_i :\Sym^i T_x\Sigma\to\hom(\Lambda^2T_x\Sigma,\mathbb{R}^p), i=1,\ldots,\alpha\big)$, where the \emph{pure} $\alpha$-jet $S_\alpha$ is given as $$S_\alpha(X_1,\ldots,X_\alpha)(Y\wedge Z) = R_{\alpha+1}(X_1,\ldots,X_\alpha,Y) (Z) - R_{\alpha+1}(X_1,\ldots,X_\alpha,Z)(Y).$$
	Again, going back to the commutative diagram \hyperref[cd:jetLifting:surjectivityOfLambda]{$(*)$}, we have
	$$\Delta_d^{(\alpha)}\big(x,\lambda\circ F, R_i\big)
	= \Delta_d^{(\alpha)}\circ\Delta_\lambda^{(\alpha+1)}(x,y,P_i,Q)
	= \Delta_{d\lambda}^{(\alpha)}\circ p^{\alpha+2}_{\alpha+1}(x,y,P_i,Q)
	= \Delta_{d\lambda}^{(\alpha)}(x,y,P_i)
	= 0,$$
	and so in particular, $S_\alpha = 0$. But then we readily see that $R_{\alpha+1}^\prime$ is a symmetric tensor.
	
	Let us now fix some basis $\{\partial_1,\ldots,\partial_{k+1}\}$ of $T_x \Sigma$ so that, $T_x\Sigma = \langle \partial_1, \ldots, \partial_{k+1} \rangle$, where $\dim\Sigma = k+1$. Then, we have the standard basis for the symmetric space $\Sym^{\alpha+2}T_x\Sigma$, so that
	$$\Sym^{\alpha+2} T_x\Sigma = \Big\langle \partial_J := \partial_{j_1}\odot \cdots\odot \partial_{j_{\alpha+2}} \;\Big|\; J=(1\le j_1 \le \cdots \le j_{\alpha+2} \le k+1) \Big\rangle.$$
	Then for each tuple $J=(j_1,\ldots,j_{\alpha+2})$, we see that the \emph{only} equation involving $Q(\partial_J)$ is
	$$0 = R_{\alpha+1}(\partial_1,\ldots,\partial_{j_{\alpha+1}})(\partial_{j_{\alpha+2}}) = \lambda \circ Q(\partial_J) + \text{terms with $P_i$.}$$
	This is an affine equation in $Q(\partial_J)\in T_y M$, which admits solution since $\lambda|_y:T_y M \to \mathbb{R}^p$ has full rank. Thus, we have solved $Q$.
	
	This concludes the proof that $p^{\alpha+2}_{\alpha+1}(\mathcal{R}_{\alpha+1}) = \bar{\mathcal{R}}_\alpha$. Since $Q$ is solved from an affine system of equation, it is immediate that the fiber $\big(p^{\alpha+2}_{\alpha+1}\big)^{-1}\big(x,y,P_i\big)$ is affine in nature. In fact, we see that the projection is an affine fiber bundle. Furthermore, since $\lambda=(\lambda^s)$ has full rank at each point, we can get lifts of sections over a fixed contractible chart $O\subset \Sigma$, where we may choose some coordinate vector fields as the basis for $T\Sigma|_O$.
\end{proof}

\begin{proof}[Proof of \autoref{lemma:jetLiftingIsoCont:OmegaRegularity}]
	We have the following commutative diagram,
	\[\begin{tikzcd}
		\bar{\mathcal{R}}_{\alpha+1}\arrow[hookrightarrow]{r}\arrow[dashed]{d} & J^{\alpha+2}(\Sigma,M) \arrow{rr}{\Delta_\lambda^{(\alpha+1)},\;\Delta_{d\lambda}^{(\alpha+1)}} \arrow{d}[swap]{p^{\alpha+2}_{\alpha+1}} && \Omega^1(\Sigma, \mathbb{R}^p)^{(\alpha+1)} \oplus \Omega^2(\Sigma,\mathbb{R}^p)^{(\alpha+1)} \arrow{d}{p^{\alpha+1}_\alpha}[swap] {p^{\alpha+1}_\alpha} \\
		\bar{\mathcal{R}}_{\alpha} \arrow[hookrightarrow]{r} & J^{\alpha+1}(\Sigma,M) \arrow{rr}[swap]{\Delta_\lambda^{(\alpha)}, \; \Delta_{d\lambda}^{(\alpha)}} && \Omega^1(\Sigma,\mathbb{R}^p)^{(\alpha)} \oplus \Omega^2(\Sigma, \mathbb{R}^p)^{(\alpha)}
	\end{tikzcd} \label{cd:jetLifting:omegaRegularity} \tag{$**$} \]
	We have already proved that $p^{\alpha+2}_{\alpha+1}$ maps $\mathcal{R}_{\alpha+1}$ surjectively onto $\bar{\mathcal{R}}_\alpha$; since $\bar{\mathcal{R}}_{\alpha+1}\subset \mathcal{R}_{\alpha+1}$ we have that $p^{\alpha+2}_{\alpha+1}$ maps $\bar{\mathcal{R}}_{\alpha+1}$ into $\bar{\mathcal{R}}_\alpha$. We show the surjectivity.
	
	Suppose $\sigma = \big(x,y, P_i : \Sym^i T_x\Sigma\to T_y M,\;i=1,\ldots,\alpha+1\big)\in\bar{\mathcal{R}}_\alpha|_{(x,y)}$ is a given jet. We need to find out $Q:\Sym^{\alpha+2}T_x\Sigma \to T_y M$ such that, $(x,y,P_i,Q)\in\bar{\mathcal{R}}_{\alpha+1}|_{(x,y)}$. We have seen that in order to find $Q$ so that $(x,y,P_i,Q)\in \mathcal{R}_{\alpha + 1}|_{(x,y)}$, we must solve the affine system $$\lambda\circ Q = \text{terms with $P_i$,}$$
	which is indeed solvable since $\lambda$ has full rank. Now in order to find $(x,y,P_i,Q)\in \bar{\mathcal{R}}_{\alpha+1} = \bar{\mathcal{R}}_\alpha \cap \ker\Delta_{d\lambda}^{(\alpha+1)}$, we need to figure out the equations involved in $\Delta_{d\lambda}^{(\alpha+1)}$. Let us write $$\Delta_{d\lambda}^{(\alpha+1)}(x,y,P_i,Q) = \big(x, P_1^*d\lambda, R_i : \Sym^i T_x\Sigma \to \hom(\Lambda^2 T_x\Sigma, \mathbb{R}^p), i=1,\ldots,\alpha+1\big).$$
	Then the \emph{pure} $\alpha+1$-jet $R_{\alpha+1}:\Sym^{\alpha+1}T_x\Sigma\to \hom(\Lambda^2 T_x\Sigma, \mathbb{R}^p)$ is the only expression that involves $Q$. In fact, we have that $R_{\alpha+1}$ is given as,
	\begin{align*}
		R_{\alpha+1}(X_1,\ldots,X_{\alpha+1})(Y\wedge Z) &= d\lambda\big(Q(X_1,\ldots,X_{\alpha+1},Y), P_1(Z)\big)\\
		&\qquad\qquad + d\lambda\big(P_1(Y), Q(X_1,\ldots,X_{\alpha+1},Z)\big)\\
		&\qquad\qquad + \text{terms involving $P_i$ with $i\ge 2$.}
	\end{align*}
	Now, looking at commutative diagram \hyperref[cd:jetLifting:omegaRegularity]{$(**)$}, we have
	\begin{align*}
		(x,y, P_1^*d\lambda, R_i, i=1,\ldots,\alpha) &= p^{(\alpha+1)}_{\alpha}\circ \Delta_{d\lambda}^{(\alpha+1)}(x,y,P_i,Q) \\
		&= \Delta_{d\lambda}^{(\alpha)} \circ p^{\alpha+2}_{\alpha+1}(x,y,P_i,Q) \\
		&= \Delta_{d\lambda}^{(\alpha)}(x,y,P_i) \\
		&= 0.
	\end{align*}
	That is, we have $R_i=0$ for $i=1,\ldots,\alpha$. In order to find $Q$ such that $R_{\alpha+1}=0$, let us fix some basis $\{\partial_1,\ldots,\partial_{k+1}\}$ of $T_x\Sigma$, where $\dim\Sigma = k+1$. Then we have the standard basis for the symmetric space $\Sym^{\alpha+2} T_x\Sigma$, so that, $$\Sym^{\alpha+2} T_x\Sigma := Span \Big\langle \partial_J = \partial_{j_1}\odot \cdots\odot \partial_{j_{\alpha+2}} \;\Big|\; J=(1\le j_1 \le \cdots \le j_{\alpha+2} \le k+1)\Big\rangle.$$
	Now for any tuple $J$ and for any $1\le a < b \le k+1$, we have the equation involving the tensor $Q$ given as,
	\[0 = R_{\alpha+1}(\partial_J) (\partial_a\wedge\partial_b) = d\lambda\big(Q(\partial_{J+a}), P_1(\partial_b\big)\big) + d\lambda\big(P_1(\partial_a), Q(\partial_{J+b}\big) + \text{terms with $P_i$ for $i \ge 2$,}\]
	where $J+a$ is the tuple obtained by ordering $(j_1,\ldots,j_{\alpha+2},a)$. Now observe that $$a < b \Rightarrow J+a \prec J+b,$$
	where $\prec$ is the lexicographic ordering on the set of all ordered $\alpha+2$ tuples. We then treat the above equation as $$\Big(\iota_{P_1(\partial_a)} d\lambda \Big) \circ Q(\partial_{J+b}) = \Big(\iota_{P_1(\partial_b)} d\lambda\Big) \circ Q(\partial_{J+a}) + \text{terms with $P_1$.}$$
	
	Thus, we have identified the defining system of equations for the tensor $Q$ given as follows:
	\[\left\{
	\begin{aligned}
		\lambda\circ Q(\partial_I) &= \text{terms with $P_i$}, \quad \text{for each $\alpha+2$ tuple $I$}\\
		\iota_{P_1(\partial_a)} d\lambda \circ Q(\partial_{J+b}) &= \iota_{P_1(\partial_b)} d\lambda \circ Q(\partial_{J+a}) + \text{terms with $P_i$},\\
		&\qquad\qquad \text{for each $\alpha+1$-tuple $J$ and $1\le a < b\le k+1$}
	\end{aligned}\tag{$\dagger$} \label{eqn:omegaRegularSystem}
	\right.\]
	We claim that this system can be solved for each $Q(\partial_I) \in T_y M$ in a \emph{triangular} fashion, using the ordering $\prec$ on the tuples. Indeed, first observe that for the $\alpha+2$-tuple $\hat I = (1,\ldots,1)$, which is the \emph{least} element in the order $\prec$, the only subsystem involving $Q(\partial_{\hat I})$ in the system (\ref{eqn:omegaRegularSystem}) is $$\lambda \circ Q(\partial_{\hat I}) = \text{terms with $P_i$,}$$
	which is solvable for $Q(\partial_{\hat I})$ as $\lambda$ has full rank. Next, for some $\alpha+2$-tuple $I$ with $\hat I \preceq I$, inductively assume that $Q(\partial_{I^\prime})$ is solved from (\ref{eqn:omegaRegularSystem}) for each $\alpha+2$-tuple $I^\prime \prec I$. Then, the subsystem involving $Q(\partial_I)$ in (\ref{eqn:omegaRegularSystem}) is given as
	\[\left\{
		\begin{aligned}
			\lambda\circ Q(\partial_I) &= \text{terms with $P_i$}, \quad \text{for each $\alpha+2$ tuple $I$}\\
			\iota_{P_1(\partial_a)} d\lambda \circ Q(\partial_I) &= \text{terms with $P_i$ and $Q(\partial_{I^\prime})$ with $I^\prime \prec I$},\\
			&\qquad\qquad \text{for $1 \le a < b \le k+1$, with $b\in I$.}
		\end{aligned}\tag{$\dagger_I$} \label{eqn:omegaRegularSystem:I}
	\right.\]
	From the induction hypothesis, the right-hand side of this affine system consists of known terms. Now, it follows from the $\Omega$-regularity condition that for any collection of independent vectors $\{v_1,\ldots,v_r\}$ in $T_x\Sigma$, the collection of $1$-forms $$\big\{\iota_{P_1(v_i)} d\lambda^s|_{\mathcal{D}_y}, \quad 1\le i \le r,\; 1\le s\le p\big\}$$
	are independent. As $\mathcal{D}$ is given as the common kernel of $\lambda^1,\ldots,\lambda^p$, we see that this is equivalent to the following non-vanishing condition: 
	$$\Big(\bigwedge_{s=1}^p \lambda^s\Big)\wedge \bigwedge_{i=1}^r\Big(\iota_{P_1(v_1)}d\lambda^1\wedge\ldots\wedge\iota_{P_1 (v_i)}d\lambda^s\Big) \ne 0.$$
	But then clearly, the subsystem (\ref{eqn:omegaRegularSystem:I}) is a \emph{full rank} affine system, allowing us to solve for $Q(\partial_I)$. Proceeding in this triangular fashion, we solve the tensor $Q$ from (\ref{eqn:omegaRegularSystem}). Clearly, the solution space for $Q$ is contractible since at each stage we have solved an affine system.
	
	We have thus proved that $p^{\alpha+2}_{\alpha+1}:\bar{\mathcal{R}}_{\alpha+1}|_{(x,y)}\to\bar{\mathcal{R}}_{\alpha}|_{(x,y)}$ is indeed surjective, with contractible fiber. In fact, the algorithmic nature of the solution shows that, if $O\subset\Sigma$ is a contractible chart, then we can obtain the lift of any section of $\bar{\mathcal{R}}_\alpha|_O$ to $\bar{\mathcal{R}}_{\alpha+1}$, along $p^{\alpha+2}_{\alpha+1}$. This concludes the proof.
\end{proof}

\begin{remark}\label{rmk:jetLiftingIsoContRegularityLowerDim}
	In the above proof of \autoref{lemma:jetLiftingIsoCont:OmegaRegularity}, the full strength of $\Omega$-regularity of $F$ has not been utilized. Note that, with our \emph{choice} of the ordered basis of $T_x\Sigma$, the vector $P_1(\partial_{k+1})$ does not appear in the left-hand side of the above triangular system (\ref{eqn:omegaRegularSystem}). In fact, we can prove \autoref{lemma:jetLiftingIsoCont} under the milder assumption that $\im F$ contains a codimension one $\Omega$-regular subspace, which in our case is the subspace $\langle F(\partial_1),\ldots,F(\partial_k)\rangle \subset T_x\Sigma$. This observation was used in \cite{bhowmickFat46} to prove the existence of germs of horizontal $2$-submanifolds in a certain class of fat distribution of type $(4,6)$.
\end{remark}

\section{Fat Distributions of corank-$2$ and their Degree}\label{sec:fatAndDegree}
In this section, we recall the preliminaries of the fatness of distribution and then introduce a notion called `degree' on the class of corank-$2$ fat distributions.

\subsection{Fat Tuple}
Let us first study fatness from an algebraic viewpoint. Consider a tuple $(D,\omega^1,\omega^2)$, where $D$ is a vector space equipped with a pair of linear symplectic forms $\omega^1,\omega^2$. We shall denote the pair $(\omega^1,\omega^2)$ by $\Omega$. Next, define an isomorphism $A:D\to D$ by
\begin{equation}\label{eqn:connectingAutomorphism}
	\omega^1(u,Av) = \omega^2(u,v), \quad \text{for $u,v\in D$.}
\end{equation}
We shall refer to $A$ as the \emph{connecting automorphism} between $\omega^1$ and $\omega^2$. For any subspace $V\subset D$ denote, $$V^{\perp_i} = \big\{w\in D \; \big|\; \omega_i(v,w) = 0,\forall v\in V\big\},\quad i=1,2, \qquad V^\Omega = V^{\perp_1}\cap V^{\perp_2}.$$

\begin{obs}\label{obs:fatTupleBasic}
	For any subspace $V\subset D$ we have the following.
	\begin{enumerate}
		\item \label{obs:fatTupleBasic:1} $V^{\perp_2} = \big(AV\big)^{\perp_1}, \quad V^{\perp_1} = A\big(V^{\perp_2}\big)$.
		\item \label{obs:fatTupleBasic:2} $V^\Omega = (V+AV)^{\perp_1} = (V+A^{-1}V)^{\perp_2}$.
		\item \label{obs:fatTupleBasic:3} The subspace $V^\Omega$ only depends on the linear span of the $2$-forms $\omega^1,\omega^2$.
	\end{enumerate}
\end{obs}

\begin{defn}
	A subspace $V\subset D$ is called $\Omega$-\emph{isotropic} if $V\subset V^\Omega$.
\end{defn}
	
\begin{defn}
	A subspace $V\subset D$ is called $\Omega$-\emph{regular} if the linear map,
	\begin{equation}\label{eqn:omegaRegularLinear}
		\begin{aligned}
			D &\to \hom(V,\mathbb{R}^2)\\
			\xi &\mapsto \big(\iota_\xi \omega^1|_V, \iota_\xi \omega^2|_V\big)
		\end{aligned}
	\end{equation}
	is surjective (compare \autoref{defn:contOmegaRegular}).
\end{defn}

We have the following characterization of regularity.
\begin{prop}\label{prop:fatTupleOmegaRegular}
	For a subspace $V\subset D$, the following statements are equivalent.
	\begin{enumerate}
		\item \label{prop:fatTupleOmegaRegular:1} $V$ is $\Omega$-regular.
		\item \label{prop:fatTupleOmegaRegular:2} $V\cap AV = \{0\}$, i.e, $V+AV$ is a direct sum.
		\item \label{prop:fatTupleOmegaRegular:3} $\codim V^\Omega = 2\dim V$.
	\end{enumerate}
\end{prop}
\begin{proof}
	Since $V^\Omega$ is the kernel of the map in \autoref{eqn:omegaRegularLinear}, it is clear from the definition that $V$ is $\Omega$-regular if and only if $\codim V^\Omega = 2\dim V$. This proves (\ref{prop:fatTupleOmegaRegular:1}) $\Leftrightarrow$ (\ref{prop:fatTupleOmegaRegular:3}). To prove (\ref{prop:fatTupleOmegaRegular:2}) $\Leftrightarrow$ (\ref{prop:fatTupleOmegaRegular:3}), note that $$\codim V^\Omega = \codim (V+AV)^{\perp_1} = \dim (V+AV), \quad\text{}$$ as $\omega^1$ is nondegenerate. Hence, $\codim V^\Omega = 2\dim V$ if and only if $V+AV$ is a direct sum.
\end{proof}

It is clear from the above proposition that $\Omega$-regularity of a subspace only depends on the span of the $2$-forms $\omega^1,\omega^2$.

\begin{defn}
	A tuple $(D,\omega^1,\omega^2)$ is called \emph{fat} if every one dimensional subspace of $D$ is $\Omega=(\omega^1,\omega^2)$-regular. 
\end{defn}

\begin{prop}\label{prop:fatTupleNoRealEigenvalue}
	$(D, \omega^1, \omega^2)$ is a fat tuple if and only if the connecting automorphism $A : D \to D$ has no real eigenvalue.
\end{prop}
\begin{proof}
	It follows from \autoref{prop:fatTupleOmegaRegular} (\ref{prop:fatTupleOmegaRegular:2}), that the tuple $(D, \omega^1,\omega^2)$ is fat if and only if for every $0 \ne \tau \in D$, the vectors $\{ \tau, A\tau \}$ are linearly independent. Equivalently, the tuple is fat if and only if every $0 \ne \tau \in D$ is not an eigenvector of $A$, i.e, $A$ has no real eigenvalue.
\end{proof}

\begin{defn}
	A fat tuple $(D,\omega^1,\omega^2)$ is said to have \emph{degree $d$} if the minimal polynomial of $A$ has degree $d$.
\end{defn}

\begin{prop}\label{prop:fatTupleVAV}
	Let $(D,\omega^1,\omega^2)$ be a degree $2$ fat tuple. Then, for any subspace $V$ of $D$,
	\begin{enumerate}
		\item \label{prop:fatTupleVAV:1}  $V+AV=V+A^{-1}V$.
		\item \label{prop:fatTupleVAV:2} $V^\Omega = (V+AV)^{\perp_1} = (V+AV)^{\perp_2} = (V+AV)^\Omega$.
		\item \label{prop:fatTupleVAV:3} $(V^\Omega)^\Omega = V+AV$.
		\item \label{prop:fatTupleVAV:4} If $V$ is $\Omega$-isotropic then $(V^\Omega)^\Omega$ is $\Omega$-isotropic. 
	\end{enumerate}
\end{prop}
\begin{proof}
	Since the minimal polynomial of $A$ is of degree 2, it follows that $A^{-1}=\lambda I-\mu A$ for some real numbers $\lambda, \mu$, with $\mu \ne 0$. Hence, $V+AV=V+A^{-1}V$ for any subspace $V$ of $D$, proving (\ref{prop:fatTupleVAV:1}). The proof of (\ref{prop:fatTupleVAV:2}) now follows directly from \autoref{obs:fatTupleBasic} (\ref{obs:fatTupleBasic:2}). Furthermore, (\ref{prop:fatTupleVAV:2}) implies (\ref{prop:fatTupleVAV:3}):
	$$(V^\Omega)^\Omega = (V^\Omega)^{\perp_1}\cap (V^\Omega)^{\perp_2} = V+AV.$$
	To prove (\ref{prop:fatTupleVAV:4}), let $V$ be $\Omega$-isotropic, i.e, $V\subset V^\Omega$, which implies $({V^\Omega})^\Omega \subset V^\Omega$. On the other hand, by (\ref{prop:fatTupleVAV:2}) and (\ref{prop:fatTupleVAV:3}) we have 
	$V^\Omega=(V+AV)^\Omega$ and $V+AV=(V^\Omega)^\Omega$. This proves that $(V^\Omega)^\Omega$ is $\Omega$-isotropic.
\end{proof}

\begin{remark}\label{rmk:horizontalNecessaryDimension}
	If $V \subset D$ is $\Omega$-isotropic, it follows from \autoref{prop:fatTupleVAV} that $V+AV$ is $\Omega$-isotropic and hence, $V + AV \subset V^\Omega$. If $V$ is furthermore $\Omega$-regular, then from \autoref{prop:fatTupleOmegaRegular} we have $\codim V^\Omega = 2\dim V$ and $\dim V^\Omega \ge \dim(V+AV) = 2\dim V$. Thus, if $V$ is $\Omega$-regular as well as $\Omega$-isotropic, then $\dim V \le \frac{1}{4} \dim D$.
\end{remark}

The following results will be useful later in \autoref{sec:application} when we shall discuss $h$-principle results for $K$-contact immersions in degree $2$ fat distributions.
\begin{prop}\label{prop:fatTupleExtension}
	Let $(D,\omega^1,\omega^2)$ be a degree $2$ fat tuple. Then, for any $(\omega^1,\omega^2)$-regular subspace $V\subset D$ and for any $\tau \not \in (V^\Omega)^\Omega$, the subspace $V^\prime = V+\langle \tau\rangle$ is again $(\omega^1,\omega^2)$-regular.
\end{prop}
\begin{proof}
	Let $V$ be $\Omega$-regular and $\tau \not \in (V^\Omega)^\Omega = V+AV$. Clearly, $\dim (V+AV) \cap \langle \tau, A\tau \rangle < 2$. Since the minimal polynomial of $A$ has degree $2$, both the subspaces $V+AV$ and $\langle \tau, A\tau\rangle$ are invariant under $A$. Consequently, their intersection is also invariant under $A$. Since $A$ has no real eigenvalue, this intersection \emph{cannot} be $1$-dimensional. Thus, $(V+AV)\cap  \langle \tau, A\tau \rangle = 0$; equivalently, $(V + \langle\tau\rangle) \cap (AV + \langle A \tau\rangle) = 0$. The proof then follows from \autoref{prop:fatTupleOmegaRegular} (\ref{prop:fatTupleOmegaRegular:2}).
\end{proof}

\begin{prop}\label{prop:fatTupleIsoContact}
	Let $(D,\omega^1,\omega^2)$ be a fat tuple. Suppose $V\subset D$ is symplectic with respect to $\omega^1$ and isotropic with respect to $\omega^2$. Then,
	\begin{enumerate}
		\item \label{prop:fatTupleIsoContact:1} $V$ is $\Omega$-regular.
	\end{enumerate} 
	If $(D,\omega^1,\omega^2)$ is of degree $2$, then
	\begin{enumerate}
		\item[(2)] \label{prop:fatTupleIsoContact:2}  $D = V^\Omega \oplus (V^\Omega)^\Omega$.
		\item[(3)] \label{prop:fatTupleIsoContact:3} $V^\Omega$ and $(V^\Omega)^\Omega$ are symplectic with respect to both $\omega^1,\omega^2$.
	\end{enumerate} 
\end{prop}
\begin{proof}
	To prove (\ref{prop:fatTupleIsoContact:1}), we need to show that $V\cap AV=0$, where $A$ is the automorphism defined by $\omega^1(u,Av) = \omega^2(u,v)$. Let $z \in V\cap AV$. Then there exists a $v\in V$ such that $z = Av$. Now, for any $u \in V$ we have $\omega^1(u,z) = \omega^1(u,Av) = \omega^2(u,v) = 0$, as $V$ is $\omega^2$-isotropic. Since $V$ is $\omega^1$-symplectic, we conclude that $z = 0$. Hence, $V\cap AV = 0$ and thus, $V$ is $(\omega^1,\omega^2)$-regular.

	Since $V$ is $\omega^2$-isotropic, $V$ and $AV$ are $\omega^1$-orthogonal, that is $\omega^1(V, AV)=0$ (\autoref{eqn:connectingAutomorphism}). As $(D,\omega^1,\omega^2)$ has degree $2$, we have real numbers $\lambda,\mu$ with $\mu \ne 0$ so that $A^2 = \lambda A + \mu I$. Then, for $u, v \in V$ we have
	\[\omega^1(Au, Av) = \omega^2(Au,v) = \omega^1(A^2 u, v) = \lambda \omega^1(Au,v) + \mu \omega^1(u,v) = \mu \omega^1(u,v).\]
	Hence, $V$ is $\omega^1$-symplectic implies that $AV$ is also an $\omega^1$-symplectic subspace. The above two observations imply that $V+AV$ is $\omega_1$-symplectic so that $(V+AV) \cap (V+AV)^{\perp_1} = 0$. Since $(V+AV)^{\perp_1} = (V+AV)^{\perp_2}$ by \autoref{prop:fatTupleVAV}(\ref{prop:fatTupleVAV:2}), it follows that $(V+AV)$ is $\omega^2$-symplectic as well. This proves (\hyperref[prop:fatTupleIsoContact:3]{2}) and (\hyperref[prop:fatTupleIsoContact:3]{3}) by recalling the \autoref{prop:fatTupleVAV} again.
\end{proof}

\subsection{Fat Distribution}

\begin{defn}
	A distribution $\mathcal{D}\subset TM$ is called \emph{fat} (or \emph{strongly bracket generating}) at $x\in M$ if for every nonzero local section $X$ of $\mathcal{D}$ near $x$, the set \[\big\{[X,Y]_x \; \big|\; \text{$Y$ is a local section near $x$}\big\}\] equals $T_x M$. The distribution is fat if it is fat at every point $x\in M$.
\end{defn}

In \cite{gromovCCMetric}, Gromov defines this as $1$-fatness. An important consequence of fatness is that for \emph{every} non-vanishing $\alpha$ annihilating $\mathcal{D}$, the $2$-form $d\alpha|_\mathcal{D}$ is nondegenerate. There are many equivalent ways to describe fat distributions. 
\begin{prop}\cite{montTour}\label{prop:fatDistDefinition}
	The following are equivalent.
	\begin{itemize}
		\item $\mathcal{D}$ is fat at $x \in M$.
		\item $\omega(\alpha)$ is a nondegenerate $2$-form on $\mathcal{D}_x$ for every $\alpha$ in the annihilator bundle $\Ann(\mathcal{D})$, where $\omega : \Ann(\mathcal{D}) \to \Lambda^2\mathcal{D}^*$ is the dual curvature map.
		\item Every $1$-dimensional subspace of $\mathcal{D}_x$ is $\Omega$-regular.
	\end{itemize}
\end{prop}

Fat distributions are interesting in themselves, and they have been studied in generality \cite{geHorizontalCC, raynerFat}. Fatness puts strict numerical constraints on the rank and corank of the distribution.
\begin{theorem}\cite{raynerFat, montTour}
	Suppose $\mathcal{D}$ is a rank $k$ distribution on $M$ with $\dim M=n$. If $\mathcal{D}$ is fat then the following numerical constraints hold,
	\begin{itemize}
		\item $k$ is divisible by $2$; and if $k < n-1$ then $k$ is divisible by $4$
		\item $k\ge (n-k)+1$
		\item The sphere $S^{k-1}$ admits $n-k$-many linearly independent vector fields
	\end{itemize}
	Conversely, given any pair $(k,n)$ satisfying the above, there is a fat distribution germ of type $(k,n)$.
\end{theorem}

When $\cork\mathcal{D}=1$, a fat distribution must be of the type $(2n,2n+1)$. In fact, corank-$1$ fat distributions are exactly the contact ones and hence are generic. In general, fatness is not a generic property \cite{zanetGenericRank4,montTour}. We now describe two important classes of fat distributions in corank-$2$ and $3$. These are holomorphic and quaternionic counterparts of contact structures.

\begin{example}\label{exmp:holomorphicContact}
	A \emph{holomorphic contact structure} on a complex manifold $M$ with $\dim_\mathbb{C} M = 2n + 1$ is a corank-$1$ holomorphic subbundle of the holomorphic tangent bundle $T^{(1,0)}M$, which is locally given as the kernel of a holomorphic $1$-form $\Theta$ satisfying $\Theta \wedge d\Theta^n \ne 0$. By the holomorphic contact Darboux theorem \cite{forstnericHoloLegendrianCurves}, $\Theta$ can be locally expressed as $\Theta = dz - \sum_{j=1}^n y_j dx_j$, where $(z,x_1,\ldots,x_n,y_1,\ldots,y_n)$ is a holomorphic coordinate. 
	Writing $z=z_1 + i z_2, x_j = x_{j1} + i x_{j2}, y_j = y_{j1} + i  y_{j2}$, we get $\Theta = \lambda^1 + i \lambda^2$, where 
	$$\lambda^1 = dz_1 - \sum_{j=1}^n \big(y_{j1} dx_{j1} - y_{j2}dx_{j2}\big),\quad 
	\lambda^2 = dz_2 - \sum_{j=1}^n \big(y_{j2} dx_{j1} + y_{j1}dx_{j2}\big).$$
	The distribution $\mathcal{D} \underset{loc.}{=} \ker \lambda^1 \cap \ker\lambda^2$ is a corank-$2$ fat distribution. We can explicitly define a frame $\{X_{j1},X_{j2},Y_{j1},Y_{j2}\}$ for $\mathcal{D}$ by
	$$X_{j1} = \partial_{x_{j1}} + y_{j1}\partial_{z_1} + y_{j2}\partial_{z_2}, \quad X_{j2} = \partial_{x_{j2}} - y_{j2}\partial_{z_1} + y_{j1}\partial_{z_2}, \quad Y_{j1} = \partial_{y_{j1}},\quad Y_{j2} = \partial_{y_{j2}}.$$
	They generate a finite dimensional Lie algebra, known as the complex Heisenberg algebra.
\end{example}

\begin{example} \label{exmp:quaternionicContact}
	A \emph{quaternionic contact structure}, as introduced by Biquard in \cite{biquardQuaternionContact}, on a manifold $M$ of dimension $4n+3$ is a corank-$3$ distribution $\mathcal{D}\subset TM$, given locally as the common kernel of $1$-forms $(\lambda^1,\lambda^2,\lambda^3)\in \Omega^1(M,\mathbb{R}^3)$ such that there exists a Riemannian metric $g$ on $\mathcal{D}$ and a Quaternionic structure $(J_i, i=1,2,3)$ on $\mathcal{D}$ satisfying, $d\lambda^i|_\mathcal{D}= g(J_i\_,\_)$. By a Quaternionic structure, we mean that  $J_i$ are (local) endomorphisms of $\mathcal{D}$ which satisfy the quaternionic relations: $J_1^2 = J_2^2 = J_3^2 = -1 = J_1J_2J_3$. Equivalently, there exists an $S^2$-bundle $Q\to M$ of triples of almost complex structures $(J_1,J_2,J_3)$ on $\mathcal{D}$. It is easy to see that any linear combination of a (local) quaternionic structure $\{J_i\}$, say $S = \sum a_i J_i$, satisfies $S^2 = -(\sum a_i^2) I$. Hence, for any non-zero $1$-form $\lambda$ annihilating $\mathcal{D}$, the $2$-form $d\lambda|_\mathcal{D}$ is nondegenerate, proving the fatness of the quaternionic contact structure.
\end{example}

\subsection{Corank-$2$ Fat Distribution}
We now focus on corank-$2$ fat distributions, in particular, on a specific class of such (real) distributions that has similar algebraic properties as the underlying real distribution of a holomorphic contact structure.\medskip

Given a corank-$2$ distribution $\mathcal{D}$, let us assume $\mathcal{D}\underset{loc.}{=}\ker\lambda^1\cap \ker\lambda^2$. Further, assume that $\omega_i = d\lambda^i|_\mathcal{D}$ is nondegenerate. Then we can define a (local) automorphism $A:\mathcal{D}\to \mathcal{D}$ by the following property: $$\omega_1(u,Av) = \omega_2(u,v),\qquad \forall u,v\in\mathcal{D}.$$
Explicitly, $A=-I_{\omega_1}^{-1}\circ I_{\omega_2}$, where $I_{\omega_i}:\mathcal{D}\to\mathcal{D}^*$ is defined by $I_{\omega_i}(v) = \iota_v\omega_i$ for all $v\in \mathcal{D}$. 
The following proposition characterizes corank-$2$ fat distributions, which is a direct consequence of \autoref{prop:fatTupleNoRealEigenvalue}.

\begin{prop}\label{prop:fatEigenValue}
	If $\mathcal{D}$ is fat at $x\in M$, then for some (and hence every) local defining form, the induced automorphism $A_x:\mathcal{D}_x\to \mathcal{D}_x$ has no real eigenvalue. Conversely, if $A_x$ has no real eigenvalue, then $\mathcal{D}$ is fat at $x\in M$.
\end{prop}
\begin{proof}
	The distribution $\mathcal{D}$ is fat at $x$ if and only if for any $0\ne v\in \mathcal{D}_x$ the map (see \autoref{defn:contOmegaRegular}) $$\mathcal{D}_x \ni u\longmapsto \Big(\omega^1(u,v), \; \omega^2(u,v)\Big) = \Big(\omega^1(u,v), \; \omega^1(u, Av)\Big) = -\Big(\iota_v \omega^1, \; \iota_{Av}\omega^1\Big)(u)$$
	is surjective, which is equivalent to linear independence of $\{v, Av\}$ for all $0 \ne v \in \mathcal{D}_x$. Hence, the proof follows.
\end{proof}

Now, given a corank-$2$ fat distribution $\mathcal{D}$ on $M$, we would like to assign an integer to each point $x\in M$. 

\begin{defn}\label{defn:degreeOfFatDistribution}
	Let $\mathcal{D}$ be a corank-$2$ fat distribution on $M$. Then, at each point $x\in M$, we associate a positive integer $\deg(x,\mathcal{D})$ by, $$\text{$\deg(x,\mathcal{D})$ := degree of the minimal polynomial of the automorphism $A_x:\mathcal{D}_x\to\mathcal{D}_x$,}$$ where $A$ is the connecting automorphism as above, for a pair of local $1$-forms defining $\mathcal{D}$ about the point $x$.
\end{defn}

We need to check that this notion of degree is indeed well-defined. Suppose, $$\mathcal{D} \underset{loc.}{=} \ker\lambda^1 \cap \ker\lambda^2 = \ker\mu^1\cap \ker\mu^2,$$
where $\lambda^i,\mu^i$ are local $1$-forms around $x\in M$. Then we can write $$\mu^1 = p\lambda^1 + q\lambda^2,\quad \mu^2 = r\lambda^1 + s\lambda^2,$$
for some local $p,q,r,s\in C^\infty(M)$ such that $\begin{pmatrix}
p & q \\ r & s\end{pmatrix}$ is nonsingular. Note that, $$d\mu^1|_\mathcal{D} = pd\lambda^1|_\mathcal{D} + qd\lambda^2|_\mathcal{D}, \quad d\mu^2|_\mathcal{D}= rd\lambda^1|_\mathcal{D} + s d\lambda^2|_\mathcal{D}.$$
Since $\mathcal{D}$ is fat, we get a pair of (local) automorphisms $A,B:\mathcal{D}\to\mathcal{D}$ defined by $$d\lambda^1(u,Av) = d\lambda^2(u,v),\quad\forall u,v\in\mathcal{D} \qquad\text{ and }\qquad d\mu^1(u,Bv) = d\mu^2(u,v),\quad\forall u,v\in\mathcal{D}.$$
It follows from \autoref{prop:fatEigenValue} that $A_x$ and $B_x$ have no real eigenvalue.
\begin{prop}\label{prop:fatDegMinPoly}
	The minimal polynomials of $A_x$ and $B_x$ have the same degree.
\end{prop}
\begin{proof}
	For simplicity, we drop the suffix $x$ in the proof. Note that $(pI + q A) B = r I + s A$, which gives \[B = (p I + q A)^{-1} (r I + s A),\]
	since $A$ has no real eigenvalue. Now, if $T : D \to D$ is invertible, then $T^{-1}$ can be written as a polynomial in $T$. Consequently, $B$ can be written as a polynomial in $A$. Similarly, $A$ can be written as a polynomial in $B$ as well.
	
	Next, recall that for a linear map $T:D\to D$, the degree of minimal polynomial $\mu_T$ is given by $$\deg \mu_T = \dim \Span \{ T^i, \; i\ge 0\} := \dim \langle T^i, \; i\ge 0\rangle.$$
	Now, suppose $S = \sum_{i=1}^k c_iT^i$ is some polynomial expression in $T$. But then for any $i \ge 0$ we have $S^i \in  \langle T^i, \; i\ge 0 \rangle = \langle I, T,\ldots, T^{d-1} \rangle$, where $d = \deg \mu_T$. Hence, $\deg \mu_S= \dim  \langle S^i,\; i\ge 0\rangle \le d = \deg \mu_T$. The proof then follows.
\end{proof}
In particular, we thus have that the notion of degree at a point is independent of the choice of $\lambda_1, \lambda_2$.
\begin{prop}
	Given a corank $2$ fat distribution $\mathcal{D}$ on $M$, the map $x\mapsto \deg (x,\mathcal{D})$ is lower semi-continuous.
\end{prop}
\begin{proof}
	Without loss of generality, we assume that $\mathcal{D}=\ker\lambda^1\cap\ker\lambda^2$. Suppose $d=\deg(x,\mathcal{D})$. Consider the map,
	\begin{align*}
		\phi:\mathcal{D} &\to \Lambda^d\mathcal{D}\\
		v &\mapsto v\wedge Av \wedge \ldots \wedge A^{d-1} v
	\end{align*}
	where $A:\mathcal{D}\to\mathcal{D}$ is the relating automorphism associated to $\omega^1,\omega^2$, where $\omega^i = d\lambda^i|_\mathcal{D}$. Clearly, $\phi$ is continuous and there exists $v_0\in\mathcal{D}_x$ such that $\phi(v_0)\ne 0$. Hence, $\phi_y$ must be nonzero for all $y$ in some neighborhood $U$ of $x$. Therefore, $\deg(y,\mathcal{D})\ge d$ for all $y\in U$. This proves the lower semi-continuity.
\end{proof}

\begin{obs} \label{obs:fatDegree}
	We make a few observations about the degree.
	\begin{enumerate}
		\item \label{obs:fatDegree:1} Since $A_x$ has no real eigenvalue, it follows that $\deg(x,\mathcal{D})$ is even for all $x$.

		\item \label{obs:fatDegree:2} Furthermore, $\deg(x,\mathcal{D}) \le \frac{1}{2}\rk\mathcal{D}$. Indeed, we note that the operators $A$ under consideration are \emph{skew-Hamiltonian}. Recall that an operator $T:D\to D$ on a symplectic vector space $(D,\omega)$, is skew-Hamiltonian if $(u,v) \mapsto \omega(u, Tv)$ is a skew-symmetric tensor on $D$. The observation then follows from \cite{skewHamitonianStructure}.
		
		\item \label{obs:fatDegree:3} In particular, it then follows that a fat distribution of type $(4,6)$ is always of degree $2$.
	\end{enumerate}
\end{obs}

\begin{defn}
	A corank-$2$ fat distribution $\mathcal{D}$ on $M$ is said to have degree $d$, if $d = \deg(x,\mathcal{D})$ for every $x\in M$.
\end{defn}

\begin{example} \label{exmp:holomorphicContactDeg2}
	In the example of holomorphic contact structure \autoref{exmp:holomorphicContact}, the $2$-forms $d\lambda^1|_\mathcal{D}$ and $d\lambda^2|_\mathcal{D}$ are related by $d\lambda^1(u, Jv) = - d\lambda^2(u,v)$ for $u, v\in \mathcal{D}$, where $J$ is the (integrable) almost complex structure on $TM$. Hence, the underlying real distribution is degree $2$ fat.
\end{example}

\begin{remark}
	We would like to remark here that there exist degree $2$ fat distribution germs which are not equivalent to holomorphic contact structures. The minimum dimension of a manifold admitting a corank-$2$ fat distribution is $6$, in which case all fat distributions are of degree $2$. Hence, degree $2$ fat distribution germs form an open set. By a result of Montgomery \cite{montGeneric}, generic distribution germs in this dimension do not admit a local frame that generates a finite dimensional Lie algebra. Hence, there are germs of degree $2$ fat distributions which are not equivalent to germs of holomorphic contact distributions. We also refer to a result of Cap and Eastwood \cite{capEastwoodSpecialDim6}.
\end{remark}

\section{$h$-Principle and Existence of $K$-Isocontact Immersions}
\label{sec:application}
We shall now obtain the $h$-principle for $\Omega$-regular, $K$-isocontact immersions $(\Sigma, K)\to (M,\mathcal{D})$, where $\mathcal{D}$ will be a degree $2$ fat distribution or a quaternionic contact structure.

\subsection{Isocontact Immersions into Degree $2$ Fat Distribution}\label{sec:hPrinIntoFat:isocontact}
Throughout this section, $\mathcal{D}$ is a degree $2$ fat distribution on $M$ and $K$ is a contact structure on $\Sigma$. Let us first note the following.
\begin{prop}\label{prop:isocontactIsOmegaRegular}
	Any formal isocontact immersion $F:(T\Sigma,K)\to (TM,\mathcal{D})$ satisfying the curvature condition (\autoref{defn:relCont}) is $\Omega$-regular.
\end{prop}
\begin{proof}
	Let $x\in \Sigma$ and $F_x : T_x\Sigma \to T_y M$ be the restriction of $F$ to $T_x \Sigma$. We choose some trivializations of $T\Sigma/K$ and $TM/\mathcal{D}$ near $x$ and $y$, respectively, such that $\tilde F_x$ is the canonical injection $\mathbb{R} \to \mathbb{R}\times \{0\} \subset \mathbb{R}^2$. Hence, there exist local $2$-forms $\eta, \omega^1,\omega^2$ such that $\Omega_K \underset{loc.}{=} \eta$ and $\Omega \underset{loc.}{=}(\omega^1,\omega^2)$ with respect to the trivializations, and the curvature condition $F^*\Omega|_K = \tilde F\circ\Omega_K$ translates into $$F^*\omega^1|_K = \eta, \quad F^*\omega^2|_K = 0.$$
	Since $K$ is contact, $\eta$ is nondegenerate. Hence, $V = F(K)$ is $\omega^1$-symplectic and $\omega^2$-isotropic. By \autoref{prop:fatTupleIsoContact} (\ref{prop:fatTupleIsoContact:1}) $V$ is $(\omega^1,\omega^2)$-regular and hence, $F$ is $\Omega$-regular.
\end{proof}

\begin{theorem}\label{thm:hPrinIsocontactDeg2}
	$\relICont$ satisfies the $C^0$-dense $h$-principle, provided $\rk\mathcal{D} \ge 2\rk K + 4$.
\end{theorem}
\begin{proof}
	We embed $(\Sigma,K)$ in $(\tilde\Sigma, \tilde K)$ where $\tilde \Sigma = \Sigma \times \mathbb{R}$, $\tilde K = d\pi^{-1} K = K \times \mathbb{R}$, and $\pi : \tilde \Sigma\to \Sigma$ is the  projection. We have the associated relation $\relIContTilde\subset J^1(\tilde\Sigma,M)$. In view of \autoref{thm:hPrinGenExtnHoriz}, it is enough to show that the map $ev : \relIContTilde|_O\to \relICont|_O$ induces a surjective map on sections over any contractible $O\subset \Sigma$.

	We first show that $ev$ is fiberwise surjective. Suppose $(x,y,F)$ is a jet in $\relICont$ and let $V = F(K_x) \subset \mathcal{D}_y$. Proceeding as in the proof of \autoref{prop:isocontactIsOmegaRegular}, we can show $V$ is $\omega^1$-symplectic and $\omega^2$-isotropic, with respect to a suitable choice of trivializations. Hence, by \autoref{prop:fatTupleIsoContact} (\hyperref[prop:fatTupleIsoContact:2]{2}) $V^\Omega \cap {V^\Omega}^\Omega = 0$. Now, $V$ being an $\Omega$-regular subspace, the codimension of $V^\Omega$ in $\mathcal{D}_y$ is $2\dim V = 2\rk K$ (\autoref{prop:fatTupleOmegaRegular}). Hence, it follows from the dimension condition that $\dim V^\Omega \ge 4$. So we can choose $0 \ne \tau \in V^\Omega$. Since $\tau \not \in {V^\Omega}^\Omega$, it follows from \autoref{prop:fatTupleExtension} that $V^\prime = V+\langle\tau\rangle$ is an $\Omega$-regular subspace of $\mathcal{D}_y$. Define an extension $\hat{F} : T_x\Sigma \times \mathbb{R}\to T_y M$ of $F$ by $$\hat{F}(v, t) = F(v) + t\tau \; \text{for $t\in \mathbb{R}$ and $v\in T_x\Sigma$.}$$ It is then immediate that $\hat{F}^{-1}(\mathcal{D}_y) = \tilde{K}_x$ and $\hat{F}$ is $\Omega$-regular. Furthermore, for $(v_i, t_i) \in \tilde{K}_x = K_x \oplus\mathbb{R}, i=1,2$, we have
	\begin{align*}
		\Omega\big( \hat F(v_1, t_1), \hat F(v_2,t_2)\big) 
		&= \Omega(F(v_1), F(v_2)), \; \text{as $\tau \in V^{\Omega} = \big(F(K_x)\big)^{\Omega}$}\\
		&= \tilde F \circ \Omega_{K_x}\big(v_1, v_2\big), \; \text{as $F^*\Omega|_K = \tilde F \circ \Omega_K$}\\
		&= \tilde{\hat F} \circ \Omega_{\tilde K_x} \big((v_1, t_1), (v_2, t_2)\big),
	\end{align*}
	where $\tilde{\hat F} : T\tilde \Sigma/\tilde K|_{(x,t)} \to TM/\mathcal{D}|_y$ is the map induced by $\hat F$ and $\Omega_{\tilde{K}}$ is the curvature form of $\tilde{K}$. Note that after identifying $T\tilde{\Sigma}/\tilde{K} = \pi^*(T\Sigma/K)$, we get $\Omega_{\tilde{K}} = \pi^*\Omega_K$. In other words, $\hat{F}$ satisfies the curvature condition relative to $\Omega_{\tilde K}$ and $\Omega$.

	Now suppose $(F,u) : T\Sigma \to TM$ is a bundle map representing a section of $\relICont$, with $u = \bs F : \Sigma \to M$ being the base map of $F$. It follows from the above discussion that we have two vector subbundles of $u^*TM$ defined as follows: 
	\[T\Sigma^\Omega := \bigcup_{\sigma \in \Sigma} \big(F(K_\sigma)\big)^\Omega \quad \text{and} \quad {T\Sigma^\Omega}^\Omega := \bigcup_{\sigma \in \Sigma} {\big(F(K_\sigma)\big)^\Omega}^\Omega,\]
	such that $T\Sigma^\Omega \cap {T\Sigma^\Omega}^\Omega = 0$ and $\rk T\Sigma^\Omega \ge 4$, as discussed in the previous paragraph. Then, using a local vector field $\tau$ in $T\Sigma^\Omega$, we can extend $F$ to a bundle monomorphism $\hat F: T(O\times \mathbb{R}) \to TM$, over an arbitrary contractible open set $O\subset \Sigma$. Clearly $\hat F$ is a section of $\relIContTilde|_O$. Thus, $ev : \Gamma \relIContTilde \to \Gamma \relICont$ is surjective on such $O$. The proof then follows by a direct application of \autoref{thm:hPrinGenExtnHoriz}.
\end{proof}

\subsubsection{Existence of Isocontact Immersions}
In view of \autoref{prop:isocontactIsOmegaRegular} we have the simpler description $$\relICont = \Big\{ (x,y, F) \; \Big|\; \text{$F$ is injective, \; $F^{-1}\mathcal{D}_y = K_x$, \; $F^*\Omega|_{K_x} = \tilde F\circ \Omega_K|_x$} \Big\}.$$
In order to prove the existence of a $K$-isocontact immersion, we thus need to produce a monomorphism $F:T\Sigma \to TM$ such that $F^{-1}\mathcal{D} = K$ and $F^*\Omega|_K = \tilde F \circ \Omega_K$. The existence of $F$ implies the existence of a monomorphism $G : T\Sigma/K \to TM/\mathcal{D}$. Conversely, given such a $G$ we can produce an $F$ as above, with $\tilde{F} = G$, under the condition $\rk \mathcal{D} \ge 3\rk K - 2$. Suppose $G$ covers the map $u : \Sigma \to M$. We construct a subbundle $\mathcal{F}\subset \hom( K, u^*\mathcal{D})$, where the fibers are given by $$\mathcal{F}_x = \Big\{F: K_x \to \mathcal{D}_{u(x)} \;\Big|\; \text{$F$ is injective and $F^*\Omega|_{ K_x} = G_x\circ \Omega_K$}\Big\}, \quad\text{for $x\in \Sigma$.}$$
We wish to get a global section of the bundle $\mathcal{F}$. Towards this end, we need to determine the connectivity of the fibers $\mathcal{F}_x$.

We consider the following linear algebraic setup. Let $(D,\omega^1,\omega^2)$ be a degree $2$ fat tuple with $\dim D = d$ and $A:D\to D$ be the connecting automorphism for the pair $(\omega^1,\omega^2)$. Define the subspace $R(k)\subset V_{2k}(D)$ as follows: $$R(k) = \Big\{b=(u_1,v_1,\ldots,u_k,v_k) \in V_{2k}(D) \;\Big|\; \substack{\text{$b$ is a symplectic basis for $\omega^1|_V$ and $V$ is $\omega^2$-isotropic},\\ \text{where $V = \big \langle u_i,v_i, \; i=1,\ldots,k \big\rangle$}}\Big\}.$$
We can identify the fiber $\mathcal{F}_x$ with $R(k)$, by fixing a symplectic basis of $K_x$.

\begin{lemma}\label{lemma:connectivityOfR(k)IsocontactDeg2Fat}
	The space $R(k)$ is $d-4k + 2$-connected.
\end{lemma}
\begin{proof}
	We proceed by induction on $k$. For $k=1$, $$R(1) = \Big\{(u,v) \in V_2(D) \;\Big|\; \omega^1(u,v) = 1\text{ and }\omega^2(u,v) = 0\Big\}.$$
	For fixed $u \in D$, consider the linear map
	\begin{align*}
		S_u : \langle u\rangle^{\perp_2} &\to \mathbb{R} \\
		v &\mapsto \omega^1(u,v)
	\end{align*}
	so that we have $R(1) = \bigcup_{u\in D\setminus 0} \{u\}\times S_u^{-1}(1)$. As $(D,\omega^1,\omega^2)$ is a fat tuple, every non-zero $u$ is $(\omega^1,\omega^2)$-regular and hence $\ker S_u = \langle u\rangle^{\perp_1} \cap \langle u\rangle^{\perp_2} = \langle u\rangle^{\Omega}$ is a codimension $1$ hyperplane in $\langle u\rangle^{\perp_2}$. Therefore, $S_u^{-1}(1)$ is an affine hyperplane. Thus, $R(1)$ is homotopically equivalent to the space of nonzero vectors $u$ in $D$ and so $R(1)$ is $(d-2)$-connected. Note that $d - 2 = d - 4.1 + 2$.
	
	Let us now assume that $R(k-1)$ is $d - 4(k-1) + 2 = d - 4k + 6$-connected for some $k\ge 2$. Observe that the projection map $p : V_{2k}(D)\to V_{2k-2}(D)$ maps $R(k)$ into $R(k-1)$. For a fixed tuple $b=(u_1,v_1,\ldots,u_{k-1},v_{k-1})\in R(k-1)$, the span $V = \langle u_1,\ldots,v_{k-1}\rangle$ is $\omega^1$-symplectic and $\omega^2$-isotropic. By an application of \autoref{prop:fatTupleIsoContact} (\hyperref[prop:fatTupleIsoContact:3]{3}) we have $V^\Omega$ is symplectic with respect to both $\omega^1$ and $\omega^2$. Moreover, since $A$ has degree $2$ minimal polynomial, we have $A(V+AV) = V+AV$. Consequently, it follows from \autoref{obs:fatTupleBasic} that, $$A(V^\Omega) = A\big((V+AV)^{\perp_2}\big) = \big(A(V+AV)\big)^{\perp_1} = (V+AV)^{\perp_1} = V^\Omega.$$
	Hence, $(\hat D, \omega^1|_{\hat D}, \omega^2|_{\hat D})$ is again a degree $2$ fat tuple, where $\hat D = V^\Omega$. Now, if we choose any $(u,v)\in V_2(\hat D)$, satisfying $\omega^1(u,v) =1$ and $\omega^2(u,v)=0$, it follows that $(u_1,\ldots,v_{k-1},u,v)\in R(k)$. In fact, we may identify the fiber $p^{-1}(b)$ with the space $$\big\{(u,v)\in V_2(\hat D) \;\big|\; \text{$\omega^1(u,v)=1,\omega^2(u,v)=0$}\big\},$$
	which is $(\dim V^\Omega -2)$-connected as it has been already noted above. Since $V$ is $\Omega$-regular, we get that $$\dim V^\Omega = \dim (V+AV)^{\perp_1} = \dim D - 2\dim V = d - 4(k-1) = d - 4k + 4.$$
	Thus, $p^{-1}(b)$ is $\dim V^\Omega - 2 = (d - 4k + 4) - 2 = d - 4k + 2$-connected.
	
	An application of the homotopy long exact sequence to the bundle $p : R(k)\to R(k-1)$ then gives us that $$\pi_i\big(R(k)\big) = \pi_i \big(R(k-1)\big),\quad\text{for $i\le d - 4k + 2$.}$$
	By induction hypothesis we have $$\pi_i\big(R(k)\big) = \pi_i \big(R(k-1)\big) = 0, \quad\text{for $i\le d - 4k + 2$.}$$
	Hence, $R(k)$ is $d-4k + 2$-connected. This concludes the proof.
\end{proof}

\begin{remark}\label{rmk:existenceGermIsoCont}
	It follows from the proof of the above theorem that $R(k)$ is non-empty for $\dim D \ge 4k$. This implies, from the local h-principle for $\relICont$, the existence of germs of $K$-isocontact immersions in a degree $2$ fat distribution $\mathcal{D}$, provided $K$ is contact and $\rk \mathcal{D} \ge 2 \rk K$.
\end{remark}

\begin{theorem}\label{thm:existenceIsocontact}
	Any map $u:\Sigma\to M$ can be homotoped to an isocontact immersion $(\Sigma, K)\to (M,\mathcal{D})$ provided $\rk\mathcal{D} \ge \max\{2\rk K + 4, \; 3\rk K - 2\}$, and one of the following two conditions holds true:
	\begin{itemize}
		\item both $ K$ and $\mathcal{D}$ are cotrivializable.
		\item $H^1(\Sigma,\mathbb{Z}_2) = 0 = H^2(\Sigma, \mathbb{Z})$.
	\end{itemize}
	Furthermore, the base level homotopy can be made arbitrary $C^0$-close to $u$.
\end{theorem}
\begin{proof}
	Suppose $u:\Sigma\to M$ is any given map. We first observe the implication of the second part of the hypothesis. If both $ K$ and $\mathcal{D}$ are given to be cotrivializable, then there exists an injective bundle morphism $G: T\Sigma/ K \to u^*TM/\mathcal{D}$. In general, the obstruction to the existence of a non-vanishing section of the rank $2$ bundle $E = \hom(T\Sigma/K, u^*TM/\mathcal{D})$ is the Euler class $e(E) \in H^2(\Sigma,\mathbb{Z})$, provided $E$ is orientable \cite[Remarks 6.2, pg. 301]{husemollerFiberBundle}. The obstruction to the orientability of the bundle $E$ is determined by the first Stiefel-Whitney class $w_1(E) \in H^1(\Sigma, \mathbb{Z}_2)$. Hence, with $H^1(\Sigma,\mathbb{Z}_2) = 0 = H^2(\Sigma, \mathbb{Z} )$, we have the required bundle map.
	
	Now, for a fixed monomorphism $G$, we construct the fiber bundle $\mathcal{F}=\mathcal{F}(u,G)\subset \hom( K, u^*TM)$ as discussed above. By \autoref{lemma:connectivityOfR(k)IsocontactDeg2Fat}, the fibers of $\mathcal{F}$ are $d-4k+2$ connected, where $\rk\mathcal{D}= d$ and $\rk K = 2k$. From the hypothesis we have, $$\rk\mathcal{D} \ge 3\rk K - 2 = 6k - 2 \;\Leftrightarrow\; d - 4k + 2 \ge 2k = \dim \Sigma -1.$$
	Hence, we have a global section $\hat F\in\Gamma\mathcal{F}$, which defines a formal, $ K$-isocontact immersion $F : T\Sigma\to u^*TM$ covering $u$, satisfying $F|_\mathcal{D}= \hat F$ and $\tilde{F} = G$. The proof now follows from a direct application of \autoref{thm:hPrinIsocontactDeg2}, since $\rk \mathcal{D}\ge 2 \rk K + 4$ by the hypothesis.
\end{proof}

\autoref{thm:mainTheoremIsoContDeg2}, stated in the introduction, is a restatement of \autoref{thm:hPrinIsocontactDeg2} and \autoref{thm:existenceIsocontact}.\medskip

\subsection{Horizontal Immersions into Degree $2$ Fat Distribution}\label{sec:hPrinIntoFat:horiz}
\begin{theorem}\label{thm:hPrinHorizImmFatDeg2}
	Suppose $\mathcal{D}\subset TM$ is a degree $2$ fat distribution on a manifold $M$ and $\Sigma$ is an arbitrary manifold. Then $\relHor$ satisfies the $C^0$-dense $h$-principle provided, $\rk\mathcal{D} \ge 4\dim\Sigma$.
\end{theorem}
\begin{proof}
	Given $\Sigma$, we consider the manifold $\tilde \Sigma = J^1(\Sigma, \mathbb{R})$, endowed with its canonical contact structure $K$. Note that $\dim\tilde\Sigma = 2\dim\Sigma +1$ and $\Sigma$ is canonically embedded as a Legendrian submanifold of $\tilde \Sigma$. Hence, any $K$-isocontact immersion $\tilde \Sigma \to M$ restricts to a horizontal immersion $\Sigma \to M$. We consider the relation $\relIContTilde \subset J^1(\tilde \Sigma, M)$ consisting of formal maps $T\tilde \Sigma \to TM$ inducing $K$ and satisfying the curvature condition, which are $\Omega$-regular by \autoref{prop:isocontactIsOmegaRegular}. Given \autoref{thm:hPrinGenExtnContact}, it is enough to show that the map $ev : \relIContTilde|_O \to \relHor|_O$ is surjective for contractible open sets $O \subset \Sigma$.
	
	Fix some contractible chart $O\subset \Sigma$ along with coordinates $\{x^i\}$. Then, we have a canonical choice of coordinates $\{x^i, p_i, z\}$ on $\tilde O = J^1(O,\mathbb{R}) \subset \tilde \Sigma$ so that the contact structure is given as \[K|_{\tilde O} = \ker\big(\theta := dz - p_i dx^i\big) = \Span \big\langle \partial_{p_i}, \; \partial_{x^i} + p_i \partial_z \big\rangle.\]
	Next, fix a coordinate chart $U\subset M$ and suppose $(F,u):TO \to TU$ is a bundle map representing a section of $\relHor$, with $u = \bs F$. Choose some trivialization $TM/\mathcal{D}|_U = \Span\langle e_1, e_2\rangle$ and write $\lambda : TM \to TM/\mathcal{D}$ as $\lambda = \lambda^1 \otimes e_1 + \lambda^2 \otimes e_2$. Denote $V = \im F\subset \mathcal{D}$. Since $F$ is $\Omega$-regular, the codimension of $V^\Omega = V^{\perp_1} \cap V^{\perp_2}$ in $V^{\perp_2}$ equals $\dim V$. Hence, for any subspace $V^\prime\subset V^{\perp_2}$ which is a complement to $V^\Omega$, we see that $\big(S := V\oplus V^\prime, d\lambda^1|_{S}\big)$ is a symplectic bundle. Our goal is to get a complement $V^\prime$ such that $S = V\oplus V^\prime$ is $\omega^2 = d\lambda^2|_\mathcal{D}$-isotropic.
	
	First, we get an almost complex structure $J : \mathcal{D} \to \mathcal{D}$ so that \[g : (u,v) \mapsto \omega^2(u,Jv),\quad u,v\in\mathcal{D}\] is a nondegenerate symmetric tensor. Such a compatible $J$ always exists and then $\omega^2$ is $J$-invariant. Since $\omega^2(V^\Omega, {V^\Omega}^\Omega) = 0$, we have $\mathcal{D} = V^\Omega \oplus_g J\big({V^\Omega}^\Omega\big)$  by \autoref{prop:fatTupleVAV}. Take $V^{\prime} = V^{\perp_2} \cap J\big({V^\Omega}^\Omega\big)$. A dimension counting argument then gives us $V^{\perp_2} = V^\Omega \oplus V^{\prime}$. Since both $V$ and $V^\prime$ are $\omega^2$-isotropic and also $\omega^2(V, V^\prime) = 0$, we have $S := V \oplus V^\prime$ is $\omega^2$-isotropic.
	
	Now, $V\subset S$ is $d\lambda^1$-Lagrangian. Consider the frame $V = \Span \langle X_i := F(\partial_{x^i}) \rangle$ and extend it to a symplectic frame $\langle X_i, Y_i \rangle$ of $(S, d\lambda^1|_S)$ so that the following holds: \[d\lambda^1(X_i, X_j) = 0 = d\lambda^1(Y_i, Y_j), \quad d\lambda^1(X_i, Y_j) = \delta_{ij}.\]
	Define the extension map $\tilde F : T\tilde O \to TM$ as follows: \[\tilde F(\partial_{x^i} + p_i \partial_z) = F(\partial_{x^i}) = X_i, \quad \tilde F(\partial_{p_i}) = Y_i, \quad F(\partial_z) = e_1.\]
	Clearly, $\tilde F$ induces $K$ from $\mathcal{D}$ and satisfies the curvature condition \[\tilde F^*d\lambda^1|_K = d\theta|_K, \quad \tilde F^*d\lambda^2|_K = 0.\]
	But then by \autoref{prop:isocontactIsOmegaRegular}, $\tilde F$ defines a section of $\relIContTilde$ over $\tilde O$. Thus, $ev : \Gamma \relIContTilde|_\Sigma \to \Gamma \relHor$ satisfies the local extension property. The $h$-principle now follows from \autoref{thm:hPrinGenExtnHoriz}.
\end{proof}

\begin{remark}\label{rmk:horizontalOptimalRange}
	As noted in \autoref{rmk:horizontalNecessaryDimension}, we necessarily need $\rk \mathcal{D} \ge 4 \dim \Sigma$ for the existence of $\Omega$-regular $\mathcal{D}$-horizontal immersions $\Sigma \to M$. Thus, the above $h$-principle is in the optimal range.
\end{remark}

\subsubsection{Existence of Regular Horizontal Immersions}
\begin{theorem}\label{thm:existenceHorizImm}
	Suppose $\mathcal{D}\subset TM$ is a degree $2$ fat distribution. Then any $u:\Sigma\to M$ can be $C^0$-approximated by an $\Omega$-regular, $\mathcal{D}$-horizontal map provided $\rk\mathcal{D} \ge \max\big\{ 4\dim\Sigma,\; 5\dim\Sigma - 3 \big\}$.
\end{theorem}

To prove the above existence theorem, it is enough to obtain a formal $\Omega$-regular, $\mathcal{D}$-horizontal immersion, covering a given smooth map $u:\Sigma\to M$. Consider the subbundle $\mathcal{F}\subset\hom(T\Sigma, u^*TM)$, where the fibers are given by
$$\mathcal{F}_x =\Big\{ F:T_x\Sigma\to \mathcal{D}_{u(x)} \quad \Big|\quad \text{$F$ is injective, $\Omega$-regular and $\Omega$-isotropic} \Big\}, \quad x\in\Sigma.$$
We need to show that $\mathcal{F}$ has a global section. Suppose $(D,\omega^1,\omega^2)$ is a degree $2$ fat tuple with $\dim D = d$ and let $V_k(D)$ denote the space of $k$-frames in $D$. Note that the fibers $\mathcal{F}_x$ can be identified with the subset $R(k)$ of $V_k(D)$ defined by
$$R(k) = \Big\{(v_1,\ldots,v_k) \in V_k(D)\;\Big|\; \text{the span $\langle v_1,\ldots,v_k\rangle$ is $\Omega$-regular and $\Omega$-isotropic}\Big\}.$$

\begin{lemma} \label{lemma:connectivityOfHorizR(k)}
	The space $R(k)$ is $d-4k+2$-connected.
\end{lemma}
\begin{proof}
	The proof is by induction over $k$. For $k=1$, we have $$R(1) = \big\{v \in D \;\big|\; \text{$v\ne 0$ and $\langle v\rangle$ is $\Omega$-regular, $\Omega$-isotropic}\big\}.$$
	Since $(D,\omega^1,\omega^2)$ is a fat tuple, \emph{every} $1$-dimensional subspace of $D$ is $\Omega$-regular as well as $\Omega$-isotropic. Thus, $R(1) \equiv D\setminus \{0\} \simeq S^{d - 1}$ and hence, $R(1)$ is $d-2$-connected. Note that, $d - 2 = d - 4.1 + 2$.
	
	Let $k\ge 2$ and assume that $R(k-1)$ is $d - 4(k-1) + 2 = d - 4k + 6$-connected. Observe that the projection map $p:V_{k}(D)\to V_{k-1}(D)$ given by $p(v_1,\ldots,v_{k})=(v_1,\ldots,v_{k-1})$ maps $R(k)$ into $R(k-1)$. To identify the fibers of $p:R(k)\to R(k-1)$,  let $b=(v_1,\ldots,v_{k-1})\in R(k-1)$ so that $V=\langle v_1,\ldots,v_{k-1}\rangle$ is $\Omega$-regular and $\Omega$-isotropic. Clearly, ${V^\Omega}^\Omega \subset V^\Omega$ since $V$ is $\Omega$-isotropic. Also, it follows from \autoref{prop:fatTupleExtension} that a tuples $(v_1,\ldots,v_{k-1},\tau) \in R(k)$ if and only if $\tau \in V^\Omega \setminus {V^\Omega}^{\Omega}$. Note that, $\dim {V^\Omega}^\Omega = 2\dim V = 2(k-1)$ and $$\codim V^\Omega = 2\dim V \;\Rightarrow\; \dim V^\Omega = d - 2(k-1).$$
	We have thus identified the fiber of $p$ over $b$: $$F(k) := p^{-1}(b) \equiv V^\Omega \setminus {V^\Omega}^\Omega \equiv \mathbb{R}^{d-2k + 2} \setminus \mathbb{R}^{2(k-1)},$$
	which is $d - 4k+ 2$-connected.	Next, consider the fibration long exact sequence associated to $p : R(k)\to R(k-1)$,
	$$\cdots\rightarrow\pi_i(F(k)) \rightarrow \pi_i(R(k)) \rightarrow \pi_i (R(k-1)) \rightarrow \pi_{i-1}(F(k))\rightarrow\cdots$$
	Since $\pi_i(F(k)) = 0$ for $i\le d - 4k + 2$, we get the following isomorphisms: $$\pi_i(R(k)) \cong \pi_i(R(k-1)), \quad\text{for $i\le d-4k+2$.}$$
	But from the induction hypothesis, $\pi_i(R(k-1)) = 0$ for $i\le d - 4k + 6$. Hence, $\pi_i(R(k)) = 0$ for $i\le d - 4k + 2$. This concludes the induction step and hence the lemma is proved.
\end{proof}

\begin{remark}\label{rmk:existenceGermHorizDeg2}
	It is clear from the above proof that $R(k) \ne \emptyset$ 	if $d \ge 4k$. Consequently, from the local $h$-principle for $\relHor$, we can conclude the existence of \emph{germs} of $k$-submanifolds horizontal to a degree $2$ fat distribution $\mathcal{D}$ provided $\rk \mathcal{D} \ge 4 \dim\Sigma$.
\end{remark}

\begin{proof}[Proof of \autoref{thm:existenceHorizImm}]
	Since $\rk \mathcal{D} \ge 5\dim\Sigma - 3$, we have a global section $F$ of $\mathcal{F}$ by \autoref{lemma:connectivityOfHorizR(k)}. Since $\rk \mathcal{D} \ge 4\dim \Sigma$ as well, \autoref{thm:hPrinHorizImmFatDeg2} implies the existence of a horizontal immersion $\Sigma \to M$.
\end{proof}

\autoref{thm:mainTheoremHorizDeg2} (stated in the introduction) now follows from \autoref{thm:hPrinHorizImmFatDeg2} and \autoref{thm:existenceHorizImm}.

\begin{corr}\label{corr:hPrinFat46}
	Given a corank-$2$ fat distribution $\mathcal{D}$ on a $6$-dimensional manifold $M$, any map $S^1 \to M$ can be homotoped to a $\mathcal{D}$-horizontal immersion.
\end{corr}
\begin{proof}
	As noted in \autoref{obs:fatDegree} (\ref{obs:fatDegree:3}), every corank-$2$ fat distribution is of degree $2$ and a formal horizontal map $S^1 \to M$ is $\Omega$-regular. The proof then follows directly from \autoref{thm:existenceHorizImm}.
\end{proof}

\begin{remark}\label{rmk:fat46}
	In \cite{bhowmickFat46}, the local $h$-principle for immersions $\mathbb{R}^2\to (M,\mathcal{D})$ horizontal to a special class of fat distribution of type $(4,6)$ was obtained. But the underlying relation for such maps need not be microflexible, and hence, it was not considered as a possible candidate for a local extension for horizontal immersions of $S^1$. The above corollary circumvents this issue by considering an extension to isocontact immersions, which is microflexible.
\end{remark}

If $\mathcal{D}$ is the underlying real distribution of a holomorphic contact structure $\Xi$ on a complex manifold $(M, J)$, where $J$ is the (integrable) almost complex structure, then because of \autoref{prop:fatTupleOmegaRegular} (\ref{prop:fatTupleOmegaRegular:2}), $\Omega$-regular immersions $\mathcal{D}$ are the same as totally real immersions. Hence, we get the following corollary to \autoref{thm:hPrinHorizImmFatDeg2}.

\begin{corr}
	Given a holomorphic contact structure $\Xi$ on $M$, there exists a totally real $\Xi$-horizontal immersion $\Sigma \to M$ provided $\rk_\mathbb{R} \Xi \ge \max\{4\dim\Sigma, 5\dim \Sigma - 3\}$.
\end{corr}

\subsection{Horizontal Immersions into Quaternionic Contact Manifolds}\label{sec:hPrinIntoFat:quaternion}
We recall the following observation by Pansu.
\begin{prop}\cite{pansuCarnotManifold} \label{prop:qContOmegaRegular}
	If $\mathcal{D}$ is a quaternionic contact structure, then any $\Omega$-isotropic subspace of $\mathcal{D}_x$ is $\Omega$-regular. Hence, every horizontal immersion is $\Omega$-regular.
\end{prop}
Because of the above result, $\relHor$ has the following simpler description
$$\relHor = \Big\{(x,y,F)\;\Big|\; \text{$F$ is injective, $F(T_x\Sigma)\subset \mathcal{D}_y$, \; $F^*\Omega = 0$} \Big\}.$$
\begin{theorem}\label{thm:hPrinHorizImmQCont}
	Suppose $\mathcal{D}\subset TM$ is a quaternionic contact structure and $\Sigma$ is an arbitrary manifold. Then $\relHor\subset J^1(\Sigma,M)$ satisfies the $C^0$-dense $h$-principle, provided $\rk\mathcal{D} \ge 4\dim \Sigma + 4$. 
\end{theorem}
\begin{proof}
	It is enough to show that under the hypothesis $\rk \mathcal{D} \ge 4\dim\Sigma + 4$, the map $ev : \relHorTilde|_O\to \relHor|_O$ is surjective on sections over any contractible open chart $O\subset \Sigma$.

	Let $(x,y,F)$ represent a jet in $\relQHor$. Then $V = \im F$ is an $\Omega$-isotropic subspace of $\mathcal{D}_y$ and so $V\subset V^\Omega$. As $V$ is $\Omega$-regular, we have $$\codim V^\Omega = \cork\mathcal{D}\times \dim V = 3\dim V.$$
	Now, from the dimension condition, we conclude that the codimension of $V$ in $V^\Omega$ is $\ge 4$. Then, for any $\tau \in V^\Omega \setminus V$ we have that $V^\prime = V+ \langle\tau\rangle$ is again isotropic. By \autoref{prop:qContOmegaRegular}, $V^\prime$ is then $\Omega$-regular as well. We can now define an extension $\tilde{F} : T_x\Sigma\oplus\mathbb{R} \to T_y M$ by $\tilde{F}(v,t) = F(v) + t\tau$ for all $v\in T_x\Sigma$ and $t\in \mathbb{R}$. Clearly $(x,y,\tilde{F})$ is then a jet in $\relQHorTilde$. Proceeding just as in \autoref{thm:hPrinIsocontactDeg2}, we can now complete the proof.
\end{proof}

\subsubsection{Existence of Horizontal Immersions}
\begin{theorem}\label{thm:existenceHorizImmQCont}
	Let $\mathcal{D}$ be a quaternionic contact structure on $M$. Then any map $u:\Sigma\to M$ can be homotoped to a $\mathcal{D}$-horizontal immersion provided, $\rk\mathcal{D} \ge \max \{4\dim\Sigma + 4, \;5\dim\Sigma - 3\}$. Furthermore, the homotopy can be made arbitrarily $C^0$-small.
\end{theorem}
\begin{proof}
	The proof is similar to that of \autoref{thm:existenceHorizImm}; in fact it is simpler since $\Omega$-regularity is automatic by \autoref{prop:qContOmegaRegular}. Given a map $u:\Sigma \to M$, we consider the subbundle $\mathcal{F} \subset \hom(T\Sigma, u^*TM)$ with the fibers given as $$\mathcal{F}_x = \Big\{F:T_x \Sigma \to \mathcal{D}_{u(x)} \;\Big|\; \text{$F$ is injective and $\Omega$-isotropic}\Big\}, \quad x\in\Sigma.$$
	Clearly, a global section of $\mathcal{F}$ is precisely a formal $\mathcal{D}$-horizontal immersion covering $u$. A choice of a frame of $T_x\Sigma$ lets us identify $\mathcal{F}_x$ with the space
	$$R(k) = \Big\{(v_1,\ldots,v_k) \in V_k(\mathcal{D}_x) \;\Big|\; \text{the span $\langle v_1,\ldots,v_k \rangle \subset \mathcal{D}_x $ is $\Omega_x$- isotropic}\Big\},$$
	where $V_k(D)$ is the space of $k$-frames in a vector space $D$. A very similar argument as in \autoref{lemma:connectivityOfHorizR(k)} gives us that the space $R(k)$, and consequently the fiber $\mathcal{F}_x$, is $\rk \mathcal{D} - 4k + 2$-connected. The proof \autoref{thm:existenceHorizImmQCont} then follows exactly as in \autoref{thm:existenceHorizImm}.
\end{proof}

We can now prove \autoref{thm:mainTheoremHorizQuat} from \autoref{thm:hPrinHorizImmQCont} and \autoref{thm:existenceHorizImmQCont}.

\begin{remark}\label{rmk:existenceGermHorizQuatContact}
	As in the previous two cases, we can deduce the existence of \emph{germs} of horizontal $k$-submanifolds to a given quaternionic contact structure $\mathcal{D}$ provided $\rk \mathcal{D} \ge 4\dim\Sigma$.
\end{remark}

\subsection{Isocontact Immersions into Quaternionic Contact Manifolds}
\begin{theorem}\label{thm:hPrinIsocontactQCont}
	Suppose $\mathcal{D}$ is a quaternionic contact structure on a manifold $M$ and $K$ is a contact structure on $\Sigma$. Then, $\relICont$ satisfies the $C^0$-dense $h$-principle provided $\rk \mathcal{D} \ge 4\rk K + 4$.
\end{theorem}
\begin{proof}
	The proof is very similar to that of \autoref{thm:hPrinIsocontactDeg2}. Suppose $F:T_x\Sigma \to T_y M$ represents a jet in $\relICont$. Suitably choosing trivializations near $x$ and $y$, we may assume that the induced map $\tilde{F} :T\Sigma/K \hookrightarrow TM/\mathcal{D}$ is the canonical injection $\mathbb{R} \to \mathbb{R}\times \{0\} \subset \mathbb{R}^3$. In particular, there exists local $2$-forms $\eta, \omega^i, i=1,2,3$ so that
	\[\Omega_K \underset{loc.}{=} \eta \quad \text{and} \quad \Omega \underset{loc.}{=} (\omega^1,\omega^2,\omega^3).\]
	And furthermore, we have a quaternionic structure $(J_1, J_2, J_3)$ so that $g(J_i u, v) = \omega^i(u,v)$ for all $u,v\in\mathcal{D}$ and for each $i=1,2,3$. Here $g$ is a Riemannian metric on the quaternionic contact structure $\mathcal{D}$. The curvature condition $F^*\Omega|_{K} = \tilde{F} \circ \Omega_K$ translates into \[F^*\omega^1|_K = \eta, \quad F^*\omega^2|_K = 0 = F^*\omega^3|_K.\]
	Since $K$ is contact, $\eta$ is nondegenerate. Consequently, $V = F(K)$ is $\omega^1$-symplectic and $\omega^2,\omega^3$-isotropic. 

	Now, for any subspace $W\subset \mathcal{D}$, we have $\mathcal{D} = W^\Omega \oplus_g \big(\sum_{i=1}^3 J_i W\big)$, indeed, \[g\left( z, \sum J_i w_i \right) = \sum \omega^i(z,w_i) = 0, \quad \forall z\in W^\Omega, \; J_i w_i \in J_i W.\]
	Consequently, $W\subset \mathcal{D}$ is $\Omega$-regular if and only if $\sum J_i W$ is a direct sum. Also, observe that, 
	\[\omega^2(u, -J_1v) = g(J_2J_1 v,u) = -g(J_3v,u) = \omega^3(u,v), \quad u,v\in\mathcal{D},\]
	and so $(\mathcal{D}_y, \omega^2|_y, \omega^3|_y)$ is a degree $2$ fat tuple with the connecting automorphism $A = -J_1$.

	As $V$ is $(\omega^2,\omega^3)$-isotropic and is $(\omega^2,\omega^3)$-regular, we get from \autoref{prop:fatTupleOmegaRegular} and \autoref{prop:fatTupleVAV} that \[V\oplus J_1V \subset V^{\perp_2}\cap V^{\perp_3} \quad \text{and} \quad \codim\big(V^{\perp_2} \cap V^{\perp_3}\big) = 2\dim V.\]
	Also, $\mathcal{D} = \big(V^{\perp_2} \cap V^{\perp_3}\big) \oplus_g \big(J_2 V + J_3 V\big)$ and hence, $\big(V^{\perp_2}\cap V^{\perp_3}\big) \cap\big(V+ J_1 V+ J_2V + J_3 V\big) = V+ J_1 V$. But then, 
	\[V^\Omega \cap \left( V+\sum J_i V \right) = V^{\perp_1} \cap \left(V^{\perp_2} \cap V^{\perp_3} \cap \left(V + \sum J_i V\right)\right) = V^{\perp_1} \cap (V+ J_1 V).\]
	Since $V$ is $\omega^1$-symplectic, we have $\mathcal{D}= V^{\perp_1} \oplus V = V^{\perp_1} + (V\oplus J_1 V)$. A dimension counting argument then gives us $\dim \big(V^\Omega \cap (V+\sum J_i V)\big) = \dim V$. But then from the hypothesis $\rk \mathcal{D} \ge 4 \dim V + 4$, we get the intersection has codimension $\ge 4$ in $V^\Omega$. Pick $\tau \in V^\Omega \setminus \big(V + \sum J_i V\big)$. We claim that $V^\prime = V+\langle\tau\rangle$ is $\Omega$-regular.

	We only need to show that $\left(\sum J_i V\right) \cap \langle J_1\tau, J_2\tau, J_3 \tau\rangle = 0$. Suppose, $z = \sum a_i J_i \tau$ is in the intersection for some $a_i\in\mathbb{R}$. Note that $ J_s\big(\sum J_i V\big) \subset V + \sum J_i V$ for each $s = 1,2,3$.  If $(a_1, a_2, a_3) \ne 0$, we have 
	\[\tau = \left(\sum a^i J_i\right)^{-1} z = \frac{-\sum a_i J_i}{\sum a_i^2} z \in \frac{-\sum a_i J_i}{\sum a_i^2} \left(\sum J_i V\right) \subset V + \sum J_i V.\]
	This contradicts our choice of $\tau \not\in V+\sum J_i V$. Hence, $z = 0$ and we have $V^\prime$ is indeed $\Omega$-regular. We can now finish the proof just as in \autoref{thm:hPrinIsocontactDeg2}.
\end{proof}

\subsubsection{Existence of Regular Isocontact Immersions}
\begin{theorem}\label{thm:existenceIsocontactQCont}
	Suppose $\mathcal{D}$ is a quaternionic contact structure on a manifold $M$ and $K$ is a contact structure on $\Sigma$. Assume that both $K$ and $\mathcal{D}$ are cotrivializable. Then, any map $u : \Sigma \to M$ can be homotoped to an $\Omega$-regular $K$-isocontact immersion $(\Sigma,K) \to (M,\mathcal{D})$ provided $\rk \mathcal{D} \ge \max \{4 \rk K + 4, 6 \rk K - 2\}$.
\end{theorem}
\begin{proof}
	Given $u : \Sigma \to M$, we can get a monomorphism $G : T\Sigma/ K \hookrightarrow u^*TM/\mathcal{D}$, since $K$ and $\mathcal{D}$ are cotrivializable. Next, we consider the bundle $\mathcal{F} \subset \hom(K, u^*\mathcal{D})$ with fibers
	\[\mathcal{F}_x = \Big\{F:K_x \to \mathcal{D}_{u(y)} \;\Big|\; \text{$F$ is injective, $\Omega$-regular and $F^*\Omega|_{K_x} = G_x \circ \Omega_K$}\Big\}.\]
	Assume $\rk K = 2k$ and $\rk \mathcal{D}=d$. Then, suitably choosing trivializations, we can identify $\mathcal{F}_x$ with the subspace $R(k) \subset V_{2k}(\mathcal{D}_y)$:
	\[R(k) = \Big\{b=(u_1,v_1,\ldots,u_k,v_k) \in V_{2k}(\mathcal{D}_y) \;\Big|\; \substack{\text{$b$ is an $\omega^1$-symplectic basis for $V:=\Span\langle u_i,v_i\rangle$,} \\ \text{$V$ is $\omega^2,\omega^3$-isotropic and $\Omega$-regular.}} \Big\}.\]
	We can check via an inductive argument similar to \autoref{lemma:connectivityOfR(k)IsocontactDeg2Fat} that $R(k)$ is $(\rk\mathcal{D} - 4\rk K + 2)$-connected (\autoref{lemma:connectivityOfIsocontactQContR(k)}). Since $\rk \mathcal{D} \ge 6\rk K - 2$, the fibers of $\mathcal{F}$ is $\dim \Sigma-1$-connected and hence, we get a global section of $\mathcal{F}$. We conclude the proof by an application of \autoref{thm:hPrinIsocontactQCont}, since $\rk \mathcal{D} \ge 4 \rk K + 4$ as well.
\end{proof}
\autoref{thm:hPrinIsocontactQCont} and \autoref{thm:existenceIsocontactQCont} implies \autoref{thm:mainTheoremIsocontactQCont}.

\begin{lemma}\label{lemma:connectivityOfIsocontactQContR(k)}
	$R(k)$ in the above theorem is $d - 8k + 2$ connected, where $\rk \mathcal{D} = d$ and $\rk K = 2k$.
\end{lemma}
\begin{proof}
	We have
	\[R(1) = \left\{(u,v)\in V_2 (\mathcal{D}_y) \;\Big|\; \omega^1(u,v)=1,\; \omega^2(u,v) = 0 = \omega^2(u,v), \; \text{$\langle u,v\rangle$ is $\Omega$-regular.}\right\}.\]
	For each $0 \ne u \in \mathcal{D}_y$ consider the map
	\begin{align*}
		S_u : u^{\perp_2} \cap u^{\perp_3} &\to \mathbb{R} \\
		v &\mapsto \omega^1(u,v)
	\end{align*}
	As argued in \autoref{thm:hPrinIsocontactQCont}, for some $v \in S_u^{-1}(1)$, the subspace $V = \langle u,v\rangle$ is $\Omega$-regular if and only if $V + J_1 V$ is a direct sum, which is equivalent to having $v \in S_u^{-1}(1) \setminus \langle u, J_1 u\rangle$. Thus, we have identified \[R(1) \equiv \bigcup_{u \in \mathcal{D}_x \setminus 0} \{u\} \times \Big(S_u^{-1}(1) \setminus \langle u,J_1 u\rangle\Big).\]
	Now, $S_u^{-1}(1)$ is a codimension $1$ affine plane in $u^{\perp_2} \cap u^{\perp_3}$ and $\langle u, J_1 u\rangle \subset u^{\perp_2} \cap u^{\perp_3}$ is transverse to $S_u^{-1}(1)$. Hence, we find out the codimension of the affine plane $S_u^{-1}(1) \cap \langle u, J_1 u\rangle$ in $u^{\perp_2} \cap u^{\perp_3}$:
	\[\codim \Big(S_u^{-1}(1) \cap \langle u,J_1 u\rangle\Big) = 1 + (\dim u^{\perp_2}\cap u^{\perp_3} - 2) = (d-2) - 1 = d - 3.\]
	But then the connectivity of $S_u^{-1}(1) \setminus \langle u, J_1 u\rangle$ is $(d-3) - (d - 2 - (d-3)) - 2 = d - 6$. Since $\mathcal{D}_x \setminus 0$ is $d - 2$-connected, we get $R(1)$ is $d-6$-connected by an application of the homotopy long exact sequence. Note that $d - 6= d- 8.1 + 2$.
	
	Let us now assume $R(k-1)$ is $d-8(k-1) + 2 = d - 8k + 10$. Now, consider the projection map $p : V_{2k}(\mathcal{D}_x) \to V_{2(k-1)}(\mathcal{D}_x)$ which maps $R(k)$ into $R(k-1)$. Say, $b = (u_1,v_1,\ldots,u_{k-1},v_{k-1}) \in R(k-1)$ and $V = \Span\langle u_i,v_i\rangle$. We show $p^{-1}(b)$ is nonempty and find out its connectivity. As in \autoref{thm:hPrinIsocontactQCont}, we must first pick $\tau \in  V^\Omega \setminus (V + \sum J_i V)$. For any such $\tau$ fixed, we set $V_\tau = V + \langle \tau \rangle$ and then choose $\eta \in (V_\tau^{\perp_2} \cap V_\tau^{\perp_3}) \setminus (V_\tau + J_1 V_\tau)$, satisfying $\omega^1(\tau , \eta) = 1$. We can check that $(u_1,v_1,\ldots,u_{k-1},v_{k-1}, \tau, \eta) \in p^{-1}(b)$. Now, let us consider the map
	\begin{align*}
		S_\tau : V_\tau^{\perp_2} \cap V_\tau^{\perp_3} &\to \mathbb{R} \\
		\eta &\mapsto \omega^1(\tau, \eta)
	\end{align*}
	Then, we have in fact identified
	\[p^{-1}(b) = \bigcup_{\tau \in V^\Omega \setminus (V+\sum J_i V)} \{u\} \times \Big(S_\tau^{-1}(1) \setminus (V_\tau + J_1 V_\tau)\Big).\]
	Since $\dim \big(V^\Omega \cap (V+\sum J_i V)\big) = \dim V$, we get the connectivity of the space of $\tau$ as \[(d - 6(k-1)) - 2(k-1) - 2 = d-8(k-1) - 2 = d - 8k + 6.\]
	On the other hand, the codimension $1$ hyperplane $S_\tau^{-1}(1)$ and $V_\tau + J_1V_\tau \subset V_\tau^{\perp_2} \cap V_\tau^{\perp_3}$ are transverse to each other. Hence, the codimension of $S_\tau^{-1}(1) \cap (V_\tau + J_1 V_\tau)$ in $V_\tau^{\perp_2} \cap V_\tau^{\perp_3}$ is \[1 + \big(\dim(V_\tau^{\perp_2} \cap V_\tau^{\perp_3}) - \dim(V_\tau + J_1 V_\tau)\big) = 1 + \big((d - 2(2k - 1)) - 2(2k - 1)\big) = d - 4(2k-1) + 1.\]
	Consequently, $S_u^{-1}(1) \cap (V_\tau + J_1 V_\tau) \equiv \mathbb{R}^{d-2(2k-1) - (d - 4(2k-1) + 1)} = \mathbb{R}^{2(2k-1) - 1}$. We get the connectivity of $S_\tau^{-1}(1) \setminus (V_\tau + J_1 V_\tau)$:
	\[(d - 2(2k - 1) - 1) - (2(2k-1) - 1) - 2 = d - 4(2k-1) - 2 = d - 8k + 2.\]
	A homotopy long exact sequence argument then gives the connectivity of $p^{-1}(b)$ as $\min\big\{d-8k + 2, d - 8k + 6\big\} = d- 8k + 2$. Then, again appealing to the exact sequence for $p : R(k) \to R(k-1)$, we get the connectivity of $R(k)$ as \[\min\{d - 8k + 2, d-8k + 10\} = d-8k + 2.\]
	This concludes the proof.
\end{proof}

\subsection{Applications in Symplectic Geometry}
A 1-form $\mu$ on a manifold $N$ is said to be a Liouville form if $d\mu$ is symplectic. Any such form defines a contact form $\theta$ on the product manifold $N\times\mathbb{R}$ by $\theta=dz-\pi^*\mu$, where $\pi:N\times \mathbb{R}\to N$ is the projection onto the first factor and $z$ is the coordinate function on $\mathbb{R}$. This construction can be extended to a $p$-tuple of Liouville forms $(\mu^1,\dots,\mu^p)$ on $N$ to obtain a corank-$p$ distribution $\mathcal{D}$ on $N\times\mathbb{R}^p$. If we denote by $(z_1,\dots,z_p)$ the global coordinate system on $\mathbb{R}^p$, then  $\mathcal{D}=\cap_{i=1}^p\ker \lambda^i$, where $\lambda^i = dz_i-\pi^*\mu^i$ and $\pi:M\times\mathbb{R}^p\to M$ is the projection map. We note that the curvature form of $\mathcal{D}$ is given as $$\Omega = \big(d\lambda^i|_\mathcal{D}\big) = \big(\pi^*d\mu^i|_\mathcal{D}\big).$$
The derivative of the projection map $\pi$ restricts to isomorphism $\pi_* : \mathcal{D}_{(x,z)}\to T_x N$ for all $(x,z)\in N\times\mathbb{R}^p$. Thus, it follows that if $(d\mu^1,\dots,d\mu^p)$ is a fat tuple on $T_x N$ for all $x\in N$, then $\mathcal{D}$ is a fat distribution.

Next, recall that given a manifold $N$ with a symplectic form $\omega$, an immersion $f: \Sigma\to N$ is called \emph{Lagrangian} if $f^*\omega = 0$. Now, $\omega = d\mu$ for some Liouville form $\mu$, a Lagrangian immersion $f:\Sigma\to N$ is called \emph{exact} if the closed form $f^*\mu$ is exact. The homotopy type of the space of exact $d\mu$-Lagrangian immersions does not depend on the primitive $\mu$, we refer to \cite{gromovBook,eliashbergBook} for the $h$-principle for exact Lagrangian immersions.

Extend this notion to $p$-tuples $(\mu^1,\ldots,\mu^p)$ of Liouville forms on $N$, if $f:\Sigma\to N$ is exact Lagrangian with respect to each $d\mu^i$, $i=1,\dots,p$, then there exist smooth functions $\phi^i$ such that $f^*\mu^i=d\phi_i$. It is easy to check that $(f,\phi^1,\dots,\phi^p) : \Sigma \to M = N \times \mathbb{R}^p$ is then a $\mathcal{D}$-horizontal immersion. Conversely, every $\mathcal{D}$ horizontal immersion $\Sigma \to M$ projects to an immersion $\Sigma\to N$ which is exact Lagrangian with respect to each $d\mu^i$.

\paragraph{\bfseries Regularity:} For immersions $f:\Sigma\to N$, we have a similar notion of $(d\mu^i)$-regularity. A subspace $V\subset T_x N$ is called \emph{$(d\mu^i)$-regular} if the map,
\begin{align*}
	\psi : T_x N &\to \hom(V,\mathbb{R}^p)\\
	\partial &\to \big(\iota_\partial d\mu^1|_V, \ldots, \iota_\partial d\mu^p|_V\big)
\end{align*}
is surjective (compare \autoref{defn:contOmegaRegular}). An immersion $f:\Sigma\to N$ is called \emph{$(d\mu^i)$-regular} if $V=\im df_\sigma$ is $(d\mu^i)$-regular for each $\sigma\in\Sigma$.
\begin{defn}\label{defn:formalBiLagrangian}
	A monomorphism $F:T\Sigma\to TN$ is said to be a \emph{formal} regular, $(d\mu^i)$-Lagrangian if for each $\sigma\in \Sigma$,
	\begin{itemize}
		\item the subspace $V=\im F_\sigma\subset T_{u(\sigma)}N$ is $(d\mu^i)$-regular subspace, and
		\item $F^*d\mu^i = 0$, that is, $V$ is $d\mu^i$-isotropic, for $i=1,\ldots,p$.
	\end{itemize}
\end{defn}

\begin{prop}\label{prop:regularExactLagrangian}
	Let $\Omega$ be the curvature of the distribution $\mathcal{D}$ on $M = N\times \mathbb{R}^p$. Then, every formal regular, $(d\mu^i)$-Lagrangian immersion lifts to a formal $\Omega$-regular $\mathcal{D}$-horizontal immersion. Conversely, any formal $\Omega$-regular $\mathcal{D}$-horizontal immersion projects to a formal regular, exact $(d\mu^i)$-Lagrangian immersion.
\end{prop}
\begin{proof}
	Suppose $(F,f):T\Sigma\to TN$ is a given formal, regular $(d\mu^i)$-Lagrangian map. Set, $u= (f,\underbrace{0,\ldots,0}_{p}):\Sigma\to M$. Then we can get a canonical lift $H:T\Sigma\to TM$ covering $u$, by using the fact that $d\pi:\mathcal{D}_{u(\sigma)}\to T_{f(\sigma)}N$ is an isomorphism. Therefore, $H$ is injective. We claim that $H$ is $\Omega$-regular and $(d\lambda^i)$-isotropic for $i=1,\ldots,p$ (in other words $\Omega$-isotropic). The isotropy condition follows easily, since, $$H^*d\lambda^i|_\mathcal{D} = H^* \pi^*d\mu^i|_\mathcal{D} = (d\pi|_\mathcal{D}\circ H)^*d\mu^i = F^*d\mu^i = 0,\qquad i=1,\ldots,p.$$
	
	To deduce the $\Omega$-regularity, observe that we have a commutative diagram,
	\[\begin{tikzcd}
		\mathcal{D}_{u(\sigma)} \arrow{r}{\phi} \arrow{d}[swap]{d\pi|_{u(\sigma)}} &\hom(\im H_\sigma,\mathbb{R}^p) \\
		T_{f(\sigma)}N \arrow{r}[swap]{\psi} & \hom(\im F_\sigma,\mathbb{R}^p) \arrow{u}[swap]{\big(d\pi|_{u(\sigma)}\big)^*}
	\end{tikzcd}\]
	where both the vertical maps are isomorphisms and the maps $\phi,\psi$ are given as $$\phi(v) = \big(\iota_v d\lambda^i|_{\im H}\big)_{i=1}^p, \; v\in \mathcal{D}_{u(\sigma)}, \qquad\text{and}\qquad \psi(w) = \big(\iota_w d\mu^i|_{\im F}\big)_{i=1}^p, \; w\in T_{f(\sigma)}N.$$
	Now, $(d\mu^i)$-regularity of $F$ is equivalent to surjectivity of $\psi$, which implies the surjectivity of $\phi$. Thus, the lift $H$ is a formal $\Omega$-regular isotropic $\mathcal{D}$-horizontal map.
	A similar argument proves the converse statement as well.
\end{proof}

In the case $p=2$, the pair $d\mu^1$ and $d\mu^2$ are related by a bundle isomorphism $A:TN\to TN$ as $d\mu^1(v,Aw)=d\mu^2(v,w)$. If for every $x\in N$, the operator $A_x$ has no real eigenvalue and the degree of the minimal polynomial of $A_x$ is $2$, then $\mathcal{D}$ is a degree $2$ fat distribution. In particular, if $N$ is a holomorphic symplectic manifold, then $\mathcal{D}$ is holomorphic contact distribution on $N\times\mathbb{R}^2$. 

\begin{theorem}\label{thm:hPrinExactBiLagrangian}
	Let $(N,d\mu^1, d\mu^2)$ as above. Then the exact Lagrangian immersions satisfy the $C^0$-dense $h$-principle, provided $\dim N \ge 4\dim\Sigma$.
\end{theorem}

The proof is immediate from \autoref{thm:hPrinHorizImmFatDeg2} and \autoref{prop:regularExactLagrangian}. Furthermore, an obstruction-theoretic argument as in \autoref{thm:existenceHorizImm} gives us the following corollary.
\begin{corr}\label{corr:existenceExactBiLagrangian}
	Suppose $(N,d\mu^1,d\mu^2)$ is as in \autoref{thm:hPrinExactBiLagrangian}. If $\dim N \ge \max \{4\dim\Sigma,\; 5\dim\Sigma-3\}$, then any $f:\Sigma\to N$ can be homotoped to a regular exact $(d\mu^1,d\mu^2)$-Lagrangian, keeping the homotopy arbitrarily $C^0$-small.
\end{corr}

\begin{remark}
	The above corollary improves upon the result in \cite{dattaSymplecticPair}, where the author proved the existence of regular, exact $(d\mu^1,d\mu^2)$-Lagrangian immersions $\Sigma\to N$, under the condition $\dim N \ge 6\dim \Sigma$.
\end{remark}

In the case $p=3$, let us assume that we have triple of symplectic forms $(d\mu^1,d\mu^2,d\mu^3)$ on a Riemannian manifold $(N,g)$. Then, we have the automorphisms $J_i : TN\to TN$ defined by $g(v, J_i w) = d\mu^i(v,w)$. If we assume that $\{J_1, J_2, J_3\}$ satisfies the quaternionic relation at each point of $N$, then $\mathcal{D}$ is quaternionic contact structure. In particular, if $N$ is hyperk\"ahler then $\mathcal{D}$ is a Quaternionic contact distribution on $N\times \mathbb{R}^3$ \cite{boyerSasakianGeometry}. Given \autoref{prop:regularExactLagrangian}, we have the direct corollary to \autoref{thm:existenceHorizImmQCont}.

\begin{corr}\label{corr:existenceTriLagrangian}
	Let $(N,g, d\mu^i, i=1,2,3)$ as above. Then, there exists an exact $(d\mu^i)$-Lagrangian immersion $\Sigma\to N$, provided $\dim N \ge \max \{4\dim \Sigma + 4, \; 5\dim \Sigma - 3\}$.
\end{corr}

\bibliographystyle{alphaurl}

\end{document}